\documentclass[12pt]{article}
\usepackage[T1]{fontenc}

\usepackage{latexsym}
\usepackage{amsmath}
\usepackage{amssymb}
\usepackage{amsthm}
\usepackage{mathtools}

\usepackage{mathrsfs}

\usepackage{graphicx}

\usepackage{tikz}

\usepackage{tikz-cd}

\usetikzlibrary{decorations.markings}
\tikzset{double line with arrow/.style args={#1,#2}{decorate,decoration={markings,%
mark=at position 0 with {\coordinate (ta-base-1) at (0,1pt);
\coordinate (ta-base-2) at (0,-1pt);},
mark=at position 1 with {\draw[#1] (ta-base-1) -- (0,1pt);
\draw[#2] (ta-base-2) -- (0,-1pt);
}}}}

\usepackage{caption}
\usepackage{subcaption}

\usepackage{float}

\usepackage{cases}

\usepackage{comment}

\usepackage[toc,page]{appendix}

\usepackage{enumitem}

\usepackage{subfiles}

\newtheorem{theorem}{Theorem}[section]

\newtheorem{corollary}[theorem]{Corollary}
\newtheorem{question}[theorem]{Question}

\newtheorem{definition}[theorem]{Definition}

\newtheorem{example}[theorem]{Example}

\newtheorem{proposition}[theorem]{Proposition}

\newtheorem{conjecture}[theorem]{Conjecture}

\newtheorem*{theorem*}{Theorem}
\newtheorem*{definition*}{Definition}
\newtheorem*{remark*}{Remark}

\DeclareMathOperator{\ch}{ch}
\DeclareMathOperator{\Ch}{Ch}
\DeclareMathOperator{\td}{td}
\DeclareMathOperator{\Ext}{Ext}

\DeclareMathOperator*{\Poly}{Polytope}

\renewcommand{\Re}{\operatorname{Re}}

\usepackage{mleftright} 
\mleftright
\def\chk#1{#1^{\smash{\scalebox{.7}[1.4]{\rotatebox{90}{\textup{\guilsinglleft}}}}}}

\newcommand{\defeq}{\coloneq}

\newcommand\restr[2]{{
\left.\kern-\nulldelimiterspace 
#1 
\vphantom{\big|} 
\right|_{#2} 
}}

\DeclarePairedDelimiter{\norm}{\lVert}{\rVert}
\DeclarePairedDelimiter{\abs}{\lvert}{\rvert}

\newcommand{\ols}[1]{\mskip.5\thinmuskip\overline{\mskip-.5\thinmuskip {#1} \mskip-.5\thinmuskip}\mskip.5\thinmuskip} 
\newcommand{\olsi}[1]{\,\overline{\!{#1}}} 
\makeatletter
\newcommand\closure[1]{
  \tctestifnum{\count@stringtoks{#1}>1} 
  {\ols{#1}} 
  {\olsi{#1}} 
}
\long\def\count@stringtoks#1{\tc@earg\count@toks{\string#1}}
\long\def\count@toks#1{\the\numexpr-1\count@@toks#1.\tc@endcnt}
\long\def\count@@toks#1#2\tc@endcnt{+1\tc@ifempty{#2}{\relax}{\count@@toks#2\tc@endcnt}}
\def\tc@ifempty#1{\tc@testxifx{\expandafter\relax\detokenize{#1}\relax}}
\long\def\tc@earg#1#2{\expandafter#1\expandafter{#2}}
\long\def\tctestifnum#1{\tctestifcon{\ifnum#1\relax}}
\long\def\tctestifcon#1{#1\expandafter\tc@exfirst\else\expandafter\tc@exsecond\fi}
\long\def\tc@testxifx{\tc@earg\tctestifx}
\long\def\tctestifx#1{\tctestifcon{\ifx#1}}
\long\def\tc@exfirst#1#2{#1}
\long\def\tc@exsecond#1#2{#2}
\makeatother

\DeclareTextFontCommand{\boldemph}{\bfseries\em}

\usepackage[indent=1.5em]{parskip}

\usepackage[
style=numeric,
sorting=nty,
maxnames=99
]{biblatex}
\usepackage{xurl}

\usepackage{hyperref}

\begin{document}

\pagenumbering{roman}

{
\title{Continuous and Discrete Asymptotic Behaviours of the J-function of a Fano Manifold}
\author{ChunYin HAU}
\date{}

\maketitle


\begin{abstract} \noindent
In this paper, we propose a condition on the coefficients of a cohomology-valued power series, which we call ``asymptotically Mittag-Leffler''. 
We show that if the $J$-function of a Fano manifold is asymptotically Mittag-Leffler, then it has the exponential growth as $t\to +\infty$. 
This provides an alternative method to compute the principal asymptotic class of a Fano manifold using the coefficients of $J$-function. 
We also verify that the $J$-function of the projective space is asymptotically Mittag-Leffler, and the property of having an asymptotically Mittag-Leffler $J$-function is preserved when taking product and hypersurface.
\end{abstract}
}

\tableofcontents

\newpage

\pagenumbering{arabic}

\pagestyle{myheadings}
\markright{Continuous and Discrete Asymptotic Behaviours of the J-function}


\section{Introduction}\label{sec:Introduction}
{


For a smooth manifold $X$, a characteristic class called the Gamma class can be defined by:
\begin{equation*}
\hat{\Gamma} \defeq \prod_{i=1}^{\dim X} \Gamma(1+\delta_i),
\end{equation*}
where the $\delta_i$ are the Chern roots of the tangent bundle.
Given an asymmetric pairing on $H^*(X)$ defined by
\begin{equation*}
[\alpha, \beta) = \frac{1}{(2\pi)^{\dim X}} \int_X (e^{\pi \mathtt{i} c_1} e^{\pi \mathtt{i} \mu} \alpha) \cup \beta,
\end{equation*}
the Hirzebruch-Riemann-Roch equation can be expressed in the form
\begin{equation*}
\begin{tikzcd}
{\chi(E,F)} \arrow[-,double line with arrow={-,-}]{r}										& {\sum_{n=0}^{\dim X} (-1)^n \dim \Ext^p (E, F)} \arrow[-,double line with arrow={-,-}]{d} \\
{[ \hat{\Gamma} \Ch (E), \hat{\Gamma} \Ch (F) )} \arrow[-,double line with arrow={-,-}]{r}	& {\int_X \ch(\chk{E}) \ch(F) \td_X}
\end{tikzcd}
\end{equation*}
where $\Ch(E) = \sum_{p=0}^{\dim X} (2 \pi \mathtt{i})^p \ch_p(E)$ is a twisted Chern character of $E$. This shows that the map $\hat{\Gamma} \Ch (-) : K(X) \rightarrow H^*(X)$ intertwines the Euler pairing $\chi$ with the pairing $[-,-)$.
The map $\hat{\Gamma}\Ch (-) : K(X) \rightarrow H^*(X)$ also appears in mirror symmetry, where it plays a crucial role in defining the integral structure for A-side quantum $D$-modules, which is mirror to the natural integral structure on the B-side quantum $D$-modules \cite{MR2553377}.


When $X$ is a Fano manifold, it is conjectured that the Gamma class of $X$ is closely related to the principal asymptotic class \cite{MR3536989, MR4384381}.
The principal asymptotic class (see $\eqref{eq:contAsym}$ below) is a cohomology class defined up to a non-zero scalar multiple. 
For a Fano manifold $X$, it can be computed by analyzing the asymptotic behaviour of the $J$-function along the anti-canonical line. 
The $J$-function of a Fano manifold is a cohomology-valued function defined as a generating function of Gromov-Witten invariants:
\begin{equation*}
J(t) = t^{c_1} \left( 1+ \sum_{\substack{ d\in \mathsf{Eff}_{\neq 0} \\ 0 \leq \ell \leq \dim H^*(X)-1}} \left\langle \frac{\phi^\ell}{1-\psi} \right\rangle_{0,1,d} \phi_\ell \; t^{\langle c_1,d \rangle} \right).
\end{equation*}
Under mild conditions on the Fano manifold, $J(t)$ exhibits the following asymptotic behaviour as $t\to +\infty$:
\begin{equation*}
J(t) \sim t^{-\frac{1}{2}\dim X} e^{Tt} A, \quad \left( T\in\mathbb{R},\; A\in H^*(X) \right) \label{eq:contAsym} \tag{$\star$}
\end{equation*}
where $A$ is the principal asymptotic class. 
The proportionality of the principal asymptotic class and the Gamma class has been verified for several classes of Fano manifolds, including Grassmanians and Fano complete intersection in projective space \cite{MR3536989, MR4384381}, Fano $3$-fold with Picard number $1$ \cite{MR3462676}, and del Pezzo surface \cite{MR4266744}. 
Notably, this phenomenon has also been verified on some monotone symplectic but not Kähler manifolds \cite{non-Kahler}.
While recent counterexamples have demonstrated that $A$ is not always proportional to $\hat{\Gamma}_X$ \cite{counterexamples}, it remains a possibility that $A$ can be represented as $\hat{\Gamma}_X \Ch (E)$ for some $K$-class $E$.


In this paper, we investigate the asymptotic behaviour of the $J$-function of a Fano manifold $X$ as $t\to +\infty$. 
We expand the $J$-function $J(t)$ into a series $J(t)=t^{c_1(X)} \sum_{m=0}^\infty J_{rm} t^{rm}$, and introduce the following condition on the coefficients:

\begin{definition*}

We say that $t^{\frac{1}{2}\dim X} J(t)$ is \boldemph{$(T,0,A)$-scaled asymptotically Mittag-Leffler} if its coefficients $J_{rm}$ satisfy the following aymptotic equivalence as $m\to\infty$:
\begin{equation*}
J_{rm} \sim \frac{T^{rm+\frac{1}{2} \dim X+c_1(X)}}{\Gamma \left( 1+rm+\frac{1}{2}\dim X+c_1(X) \right)} \cdot A. \quad \left( T\in\mathbb{R}_{>0},\; A\in H^*(X) \right) \label{eq:disAsym} \tag{$\star\star$}
\end{equation*}

\end{definition*}

We will prove that $\eqref{eq:disAsym}$ implies $\eqref{eq:contAsym}$, thus showing that the principal asymptotic class can also be computed from the asymptotic behaviour of the coefficients of the $J$-function as $m\to\infty$. 
We will also establish that $\eqref{eq:disAsym}$ is preserved under products and hypersurfaces.

\begin{theorem*}[Main Results]
If the $J$-function of a Fano manifold $X$ is asymptotically Mittag-Leffler $\eqref{eq:disAsym}$, then:
\begin{enumerate}[label=\normalfont (\arabic*)]
\item the $J$-function has the asymptotic equivalence $\eqref{eq:contAsym}$ as $t\to +\infty$ (Corollary \ref{corDisAsymtoContAym}), 
\item if another Fano manifold $Y$ also has an asymptotically Mittag-Leffler $J$-function, then the $J$-function of $X \times Y$ is also asymptotically Mittag-Leffler (Corollary \ref{corDisGammaIProduct}), 
\item if $Z$ is a degree $d$ Fano hypersurface of $X$ in the linear system $\abs{-\frac{d}{r} K_X}$ with $d=1, \ldots, r-1$ and $r=\max \{n\in\mathbb{N} \mid \frac{1}{n} c_1(X) \in H^2(X,\mathbb{Z}) \}$ is the Fano index of $X$, then the $J$-function of $Z$ is also asymptotically Mittag-Leffler (Proposition \ref{DisGammaIHypersurface}).
\end{enumerate}
\end{theorem*}

We will first define $(T,\theta,A)$-scaled asymptotic Mittag-Leffler for cohomology-valued series, then prove the part (1) above in this general setting. 
We will then verify that $J$-function of projective space is asymptotically Mittag-Leffler.
Combining this with parts (2) and (3), we conclude that the $J$-functions of products and Fano hypersurfaces of projective space are also asymptotically Mittag-Leffler. 
We will further discuss the conditions under which the $J$-function of a Fano manifold is asymptotically Mittag-Leffler, and the potential relationship between the scaling $(T,\theta,A)$ and the eigenvalues of the operator $c_1(X)\star_0: H^*(X)\rightarrow H^*(X)$. 

\begin{remark*}
Kai Hugtenburg in \cite{non-Kahler} employed a more abstract approach to compute the principal asymptotic class via the asymptotics of coefficients of the $J$-function. He introduced three conditions: for some constant $C$,
\begin{enumerate}[label=\normalfont \arabic*.]
\item the $H^0$ component of the $J$-function, i.e., $\langle [\mathsf{pt}], J(t) \rangle$, has non-negative coefficients and is ``superpolynomially peaked near $n=Ct$'',
\item the quotients of the cohomology components relative to the $H^0$ component, i.e., $\frac{\langle \mathsf{a}, J_{rm} \rangle}{\langle [\mathsf{pt}], J_{rm} \rangle}$, are ``subpolynomially increasing'' for all $\mathsf{a} \in H_*(X)$,
\item the coefficients of the $J$-function satisfy the following asymptotic equivalence as $m\to\infty$: 
\begin{equation*}
{J_{rm} \sim \langle [\mathsf{pt}], J_{rm} \rangle \left( \frac{C}{m} \right)^{c_1(X)} \cdot A}.
\end{equation*}
\end{enumerate}
Under these conditions, he proved that the $J$-function has the following asymptotic behaviour as $t\to +\infty$:
\begin{equation*}
J(t) \sim \langle [\mathsf{pt}], J(t) \rangle \cdot A.
\end{equation*}
A $(T,A)$-scaled asymptotically Mittag-Leffler $J$-function with $\langle [\mathsf{pt}], A \rangle \neq 0$ satisfies all three of Hugtenburg's conditions with $C=T/r$, and part (1) of our result aligns with his finding. 
\end{remark*}


In this paper, we assume that $0 \in \mathbb{N}$, and all cohomology spaces of $X$ are taken with complex coefficients, i.e., $H^*(X) = H^*(X, \mathbb{C})$.

In Section $2$, we provide a brief introduction to the principal asymptotic class of a Fano manifold. 
We also review the relationship between the Gamma class, the principal asymptotic class, and the limit of the $J$-function along the anti-canonical line. 
Additionally, we introduce the Riemann-Liouville integral, a generalization of the repeated integration formula, for use in Section $3$.

In Section $3$, we define $(T,\theta,A)$-scaled asymptotically Mittag-Leffler series as series whose coefficients are asymptotically equivalent to those of the Mittag-Leffler function. 
We state and prove the asymptotic behaviour an asymptotically Mittag-Leffler series exhibits as $t \to e^{\mathtt{i}\phi} \cdot \infty$ for various $\phi\in\mathbb{R}$. This result is then applied to the $J$-function of a Fano manifold. 
We provide a potential example of a Fano manifold having an asymptotic Mittag-Leffler $J$-function with $e^{\mathtt{i}\theta}\neq 1$.
We also propose a conjectural condition for a Fano manifold to have an asymptotically Mittag-Leffler $J$-function, and discuss its implications for the distribution of eigenvalues of the operator $c_1(X) \star_0:H^*(X) \rightarrow H^*(X)$.
Finally we verify that the $J$-function of projective spaces is asymptotically Mittag-Leffler in this section.

In Sections $4$ and $5$, we prove that the property of being an asymptotically Mittag-Leffler series is preserved under products and hypersurfaces. 
}


\section{Background}\label{sec:Background}
{

In this section, we will briefly review the background on defining the principal asymptotic class, including the Dubrovin connection and the $J$-function. We will then review the Gamma conjecture I, which relates the principal asymptotic class and the Gamma class. Finally, We will extend the Riemann-Liouville integral to cohomology-valued functions and establish some of its basic properties.


\subsection{Principal Asymptotic Class}

The principal asymptotic class is initially defined as a generator of a subspace of flat
sections. To define the Dubrovin connection of a Fano manifold, we introduce the Gromov-Witten invariants. These invariants virtually count the number of intersections of stable curves within the Fano manifold.



\begin{definition}[Gromov-Witten Invariant {\cite[Chapter~7]{MR1677117}}, \cite{MR1291244}]

Let $X$ be a Fano manifold. Then the \boldemph{$n$-point genus $0$ Gromov-Witten invariant with gravitational descendants} is defined as:
\begin{equation*}
\langle \psi^{k_1} \alpha_1, \psi^{k_2} \alpha_2, \ldots, \psi^{k_n} \alpha_n \rangle_{0,n,d} = \int_{[\closure{\mathcal{M}}_{0,n}(X,d)]^{\mathrm{vir}}} \prod_{i=1}^n \left( c_1(\mathcal{L}_i)^{k_i} \cup \mathrm{ev}^*_i (\alpha_i) \right)
\end{equation*}
where:
\begin{itemize}
\item $k_i\in \mathbb{N}$, $\alpha_i \in H^*(X)$ for $i=1,2,\ldots, n$, 
\item $d\in H_2(X,\mathbb{Z})$,
\item $[\closure{\mathcal{M}}_{0,n}(X,d)]^{\mathrm{vir}}$ is the virtual fundamental class of the moduli stack of $n$-pointed genus $0$ stable maps to $X$ of degree $d \in H_2(X)$,
\item $\mathrm{ev}_i: \closure{\mathcal{M}}_{0,n}(X,d) \rightarrow X$ is the evaluation map at the $i$-th marked point,
\item $\mathcal{L}_i$ is the line bundle over $\closure{\mathcal{M}}_{0,n}(X,d)$ whose fibre at a stable map is the cotangent space of the source curve at $i$-th marked point.
\end{itemize}

\end{definition}

Using Gromov-Witten invariants we can define a product structure, the quantum product, on the cohomology group. 
Let $\{\phi_\ell\}_{\ell=0}^{\dim H^*(X)-1}$ be a basis of $H^*(X)$, such that $\phi_0=1 \in H^0(X)$ and $\{\phi_j\}_{j=1}^{b^2(X)}$ is a basis of $H^2(X)$. Let $\{\phi^\ell\}_{\ell=0}^{\dim H^*(X)-1}$ be dual basis of $\{\phi_\ell\}_{\ell=0}^{\dim H^*(X)-1}$ with respect to Poincaré pairing.

\begin{definition}[Quantum Product {\cite[Chapter~8]{MR1677117}}]

Let $X$ be a Fano manifold. The \boldemph{quantum product}, $\star_\tau: H^*(X) \times H^*(X) \rightarrow H^*(X)$, for $\tau \in H^2(X)$, is defined by the formal sum
\begin{equation*}
\alpha \star_\tau \beta = \alpha \cup \beta + \sum_{\substack{d\in\mathsf{Eff}_{\ne 0}, \\0\leq \ell \leq N-1}} \langle \alpha , \beta , \phi^\ell \rangle_{0,3,d} \; \phi_\ell \; e^{\langle \tau , d \rangle} \tag{$\diamond$} \label{eq:defQuantumPro}
\end{equation*}
where $\mathsf{Eff} \subseteq H_2(X,\mathbb{Z})$ is the set of effective curve classes.

\end{definition}

The quantum product reduces to the classical cup product when $\langle \tau , d \rangle \to -\infty$ for non-zero effective class $d \in \mathsf{Eff}_{\neq 0}$. 
When $X$ is Fano, the degree axiom of Gromov-Witten invariants \cite{MR1291244} implies that $\langle \alpha , \beta , \gamma \rangle_{0,3,d} = 0$ when $\deg \alpha + \deg \beta + \deg \gamma \neq \langle c_1,d \rangle$. 
Therefore, the sum on the right-hand side of $\eqref{eq:defQuantumPro}$ is finite.

All the genus-$0$ Gromov-Witten invariants can be encoded within a connection, known as the Dubrovin conenction.
\begin{definition}[Quantum Differential Equations {\cite[Chapter~10]{MR1677117}}]\label{defDubrovinConn}

Let $X$ be a Fano manifold. Let $B$ be the trivial vector bundle with fibre $H^*(X)$ over $H^2(X) \times \mathbb{C}^\times$. Set the coordinate $(\tau,z)=(\sum_{j=1}^{b^{2}(X)} t_j \phi_j, z) \in H^2(X) \times \mathbb{C}^\times$. The \boldemph{Dubrovin connection} on $B$ can be defined by:
\begin{align*}
\nabla_{\partial_{t_j}} \varphi &= \frac{1}{z} \phi_j \star_\tau \varphi, \\
\nabla_{z\partial_{z}} \varphi &= -\frac{1}{z} c_1(X) \star_\tau \varphi + \mu(\varphi),
\end{align*}
where $\mu:H^*(X) \rightarrow H^*(X)$; the Hodge grading operator, is a linear map defined by
\begin{equation*}
\mu(\phi_\ell)=\frac{1}{2} (\deg \phi_\ell - \dim X) \phi_\ell,
\end{equation*}
and $\varphi \in H^*(X)$ is regarded as a constant section of the trivial bundle.

\end{definition}

The Dubrovin connection is a flat connection. Its fundamental solution along the $\tau$-direction, i.e., sections satisfying $\nabla_{\partial_{t_j}} \left( L(\tau,z) \alpha \right) =0$ for all $j=1, \ldots, b^2(X)$, can be given by
\begin{equation*}
L(\tau, z)\alpha \defeq e^{-\tau/z} \alpha - \hspace{-1em} \sum_{\substack{ d \in \mathsf{Eff}_{\neq 0} \\ 0 \leq \ell \leq \dim H^*(X)-1}} \hspace{-1em} \left\langle \phi^\ell, \frac{e^{-\tau/z}\alpha}{z+\psi} \right\rangle_{0,2,d} e^{\langle \tau, d \rangle} \phi_\ell,
\end{equation*}
where the second argument of the coefficients is expanded as
\begin{equation*}
\frac{1}{z+\psi}=\sum_{k=0}^\infty (-1)^k \: z^{-(k+1)} \: \psi^k.
\end{equation*}
By the linearilty of the Gromov-Witten invariants, the coefficients are a sum of the Gromov-Witten invariants with gravitational descendants.


To isolate a $1$-dimensional subspace of flat sections of the Dubrovin connection, a stability condition is imposed on the distribution of eigenvalues of the operator $c_1(X) \star_0: H^*(X)\rightarrow H^*(X)$.

\begin{definition}[Simple Rightmost Eigenvalue \cite{MR3536989, counterexamples}]\label{defRightmost}

Let $C$ be a linear endomorphism of a finite dimensional vector space. Let $\sigma(C) \subseteq \mathbb{C}$ be the set of eigenvalues of $C$. Let $T$ be an eigenvalue of $C$ having the largest real part, i.e.,
\begin{equation*}
T\in \sigma(C) \text{ such that } \Re T= \max_{\lambda \in \sigma(C)} \left( \Re \lambda \right).
\end{equation*}
We say that $C$ has a \boldemph{simple rightmost eigenvalue} if $T$ is a unique eigenvalue with that real part, i.e.,
\begin{enumerate}[label=\normalfont \arabic*.]
\item $T$ is a simple root of the characteristic polynomial of $C$, and
\item if some $T' \in \sigma(C)$ satisfies $\Re T' = \Re T$, then $T = T'$.
\end{enumerate}

\end{definition}

Note that if $C$ is a real operator, e.g., $C= c_1(X) \star_0 : H^*(X) \rightarrow H^*(X)$ for some Fano manifold $X$, and if $C$ has a simple rightmost eigenvalue $T$, then $T$ is necessarily a real number.

When $c_1(X)\star_0$ has a simple rightmost eigenvalue $T$, results from \cite{MR3536989} allow us to isolate a $1$-dimensional space consisting of flat sections with the fastest rate of decay, $e^{-Tt}$, as $t\to +\infty$. The principal asymptotic class is defined as a generator of this $1$-dimensional space. 

\begin{definition}[Principal Asymptotic Class \cite{MR3536989}]\label{defPrinClass}

Let $X$ be a Fano manifold whose $c_1(X) \star_0$ has a simple rightmost eigenvalue $T$. Parametrize the anti-canonical line by $\iota_1: \mathbb{R}_{>0} \rightarrow \mathbb{R}\cdot c_1(X) \times \{1\} \subseteq H^2(X) \times \mathbb{C}^\times$ defined by $\iota_1: t \mapsto (c_1(X) \log t ,1)$. Consider the pullback of the trivial bundle $B$ defined in Definition \ref{defDubrovinConn} by this parametrization. Then the pullback Dubrovin connection is given by
\begin{equation*}
\nabla_{t\frac{d}{dt}} = t\frac{d}{dt} + c_1(X) \star_{c_1(X) \log t}.
\end{equation*}
Define a subspace of the space of flat sections of the pullback bundle by
\begin{align*}
& \mathcal{A} \defeq \\
& \left\{ s: \mathbb{R}_{>0} \rightarrow  H^*(X) \;\middle\vert\; \nabla_{t \frac{d}{dt}} s(t) = 0, \; \norm*{e^{T t} \, s(t)}=\mathit{O}(t^m) \text{ as } t\to +\infty \text{ for some } m \right\}.
\end{align*}
The space $\mathcal{A}$ is $1$-dimentional because $c_1(X) \star_0$ has a simple rightmost eigenvalue \cite{MR3536989}. Define the \boldemph{principal asymptotic class} $A_X \in H^*(X)$ of $X$ to be a generator of $\mathcal{A}$, i.e.,
\begin{equation*}
\mathbb{C} \cdot (L(c_1(X) \log t,1) A_X) = \mathcal{A}.
\end{equation*}
Note that the principal asymptotic class is defined up to a non-zero scalar multiple $c\in \mathbb{C}^\times$.

\end{definition}

Extend the above parametrization to the complex anti-canonical line $\iota_1: \widetilde{\mathbb{C}^\times} \rightarrow H^*(X) \times \mathbb{C}^\times$, where $\widetilde{\mathbb{C}^\times}$ denotes the universal covering of $\mathbb{C}^\times$. Also parametrize the $z$-directional line by $\iota_2: z \mapsto (0,z)$. Then we have parametrizations of the anti-canonical line and the $z$-line, respectively, defined as:
\begin{equation*}
\begin{tikzcd}[row sep=0pt]
\widetilde{\mathbb{C}^\times} \arrow[r, "\iota_1"] & H^*(X)\times \mathbb{C}^\times &  & \widetilde{\mathbb{C}^\times} \arrow[r, "\iota_2"] & H^*(X)\times \mathbb{C}^\times \\
t \arrow[r, maps to] & {(c_1(X) \log t, 1)}, &  & z \arrow[r, maps to] & {(0,z)}.
\end{tikzcd}
\end{equation*}
The pullback Dubrovin connections by $\iota_1$ and $\iota_2$ are gauge equivalent after the identification $t=z^{-1}$ \cite{MR3536989}, i.e.,
\begin{equation*}
\iota_1^* \nabla = z^{\mu} \; \left( \iota_2^* \nabla \right) \; z^{-\mu}.
\end{equation*}
Therefore, information about the singularity at $z=0$ of the Dubrovin connection on the line $\tau=0$ can also be obtained by considering $t\to\infty$ on the anti-canonical line.

Computing the principal asymptotic class via the Dubrovin connection can be difficult in practice. Fortunately, it can also be computed by considering a limit involving a cohomology-valued function called the $J$-function.
\begin{definition}[$J$-function {\cite[Chapter~10]{MR1677117}}]\label{defJFunction}

Let $X$ be a Fano manifold. Then the \boldemph{$J$-function} of $X$, $J: H^*(X) \times \mathbb{C}^\times \rightarrow H^*(X)$, is defined as the following cohomology-valued function:
\begin{equation*}
J: (\tau,z) \mapsto L(\tau,z)^{-1} 1 =e^{\tau/z} \left( 1+ \hspace{-1em} \sum_{\substack{ d \in \mathsf{Eff}_{\neq 0} \\ 0 \leq \ell \leq \dim H^*(X)-1}} \hspace{-1em} \left\langle \frac{\phi^\ell}{z-\psi} \right\rangle_{0,1,d} \hspace{-1em} \phi_\ell \; e^{\langle \tau, d \rangle} \right).
\end{equation*}
On the anti-canonical line $\mathbb{R} \cdot c_1(X) \times \{1\}$ with coordinate $(c_1(X) \log t , 1), t\in\mathbb{R}_{>0}$, the $J$-function is of the form:
\begin{equation*}
J(c_1(X) \log t, 1) = J(t) = t^{c_1(X)} \left( 1+ \hspace{-1em} \sum_{\substack{ d\in \mathsf{Eff}_{\neq 0} \\ 0 \leq \ell \leq \dim H^*(X)-1}} \hspace{-1em} \left\langle \frac{\phi^\ell}{1-\psi} \right\rangle_{0,1,d} \hspace{-1em} \phi_\ell \; t^{\langle c_1(X),d \rangle} \right).
\end{equation*}
Let $r=\max \{n\in\mathbb{N} \mid \frac{1}{n} c_1(X) \in H^2(X,\mathbb{Z}) \}$ be the Fano index of $X$. Then $J(t)$ can be expanded as a power series of the form $J(t) = t^{c_1(X)} \sum_{m=0}^\infty J_{rm} t^{rm}$.

\end{definition}


By the relation between the $J$-function and the Dubrovin connection, the principal asymptotic class can also be computed by the limit of $J(t)$ when $t\to +\infty$:

\begin{proposition}[\cite{MR3536989}]\label{propContAsym}

Let $X$ be a Fano manifold. Let $J(t)$ be the $J$-function of $X$ on the anti-canonical line, as defined in Definition \ref{defJFunction}. Assume that $c_1(X) \star_0$ has a simple rightmost eigenvalue (Definition \ref{defRightmost}).
We have the asymptotic equivalence of $J(t)$ as $t\to +\infty$: 
\begin{equation*}\label{TotalContAsymp}
J(t) = t^{-\frac{1}{2} \dim X} \; e^{Tt} \cdot \left(A + \mathit{o}(1) \right),
\end{equation*}
where $T$ is the rightmost eigenvalue of $c_1(X) \star_0$, and $A$ is a principal asymptotic class of $X$.

\end{proposition}


\subsubsection*{Gamma Class and Gamma Conjecture}

It has been conjectured that the principal asymptotic class is equal to the Gamma class. We extend the definition of the Gamma function to cohomology classes. 

\begin{definition}\label{defGammaFunction}

Let $X$ be a manifold. For $\alpha\in H^*(X)$, let $\alpha=\alpha_0 + \alpha_{\geq 2}$, where $\alpha_0\in H^0(X)$ and $\alpha_{\geq 2} \in H^{\geq 2}(X)$. For $\alpha_0 \notin -\mathbb{N}$, $\Gamma(\alpha)\in H^*(X)$ is defined to be one of the following (all are equivalent):
\begin{enumerate}[label=\normalfont \arabic*.]
\item \begin{equation*} \sum_{k=0}^{\dim X} \frac{1}{k!} \restr{\frac{d^k}{dx^k} \Gamma(x)}{x=\alpha_0} \hspace{-1em} \alpha_{\geq 2}^k \quad \left( = \exp\left( \alpha_{\geq 2} \restr{\frac{d}{dx}}{x=\alpha_0} \right) \; \Gamma(x) \right) , \end{equation*}
\item for $\Re \alpha_0 > 0$, \begin{equation*} \int_0^\infty e^{-t} \: t^{\alpha-1} \; dt, \end{equation*} where $t^{\alpha-1} = e^{(\alpha-1) \log t} = t^{\alpha_0 - 1} \sum_{k=0}^{\dim X} \frac{1}{k!} \: (\log t)^k \: (\alpha_{\geq 2})^k$ is regarded as a cohomology-valued function,
\item \begin{equation*} \lim_{m\to \infty} \frac{m! \; (m+1)^\alpha}{\prod_{k=1}^m (\alpha+k)}. \end{equation*}
\end{enumerate}

\end{definition}

Definition \ref{defGammaFunction} implies that the following properties of the Gamma function naturally extend to cohomology classes:

\begin{proposition}\label{propGammaFunction}

Let $X$ be a manifold, $\alpha \in H^*(X)$. Then the Gamma function on $H^*(X)$ satisfies the following properties:
\begin{enumerate}[label=\normalfont (\arabic*)]
\item \begin{equation*} \alpha \; \Gamma(\alpha)=\Gamma(\alpha+1), \end{equation*}
\item (multiplication formula) for $m\in \mathbb{N}_{> 0}$, \begin{equation*} \prod_{k=0}^{m-1} \Gamma\left( \alpha +\frac{k}{m} \right) = (2\pi)^{\frac{m-1}{2}} \; m^{\frac{1}{2}-m \alpha} \; \Gamma(m \alpha), \end{equation*}
\item \begin{equation*} \lim_{n \to \infty} \frac{\Gamma(n+\alpha)}{\Gamma(n) \; n^\alpha} =1, \end{equation*}
\item (Stirling approximation) \begin{equation*} \lim_{n \to \infty} \frac{\Gamma(1+n+\alpha)}{(2 \pi)^\frac{1}{2} \; (n+\alpha)^\frac{1}{2} \; (n+\alpha)^{n+\alpha} \; e^{-(n+\alpha)}}=1. \end{equation*}
\end{enumerate}

\end{proposition}

The Gamma class of a smooth manifold $X$ is then defined by:

\begin{definition}[Gamma Class \cite{MR3536989}]\label{defGammaClass}

Let $X$ be a manifold, and let $\{ \delta_i \}_{i=1}^{\dim X}$ be the Chern roots of $X$. Then the \boldemph{Gamma class} of $X$, written as $\hat{\Gamma}_X$, is defined by
\begin{equation*}
\hat{\Gamma}_X \defeq \prod_{i=1}^{\dim X} \Gamma(1+\delta_i)=\exp \left( -\gamma \; c_1(X) + \sum_{k \geq 2} (-1)^k \; (k-1)! \; \zeta(k) \; \ch_k(X) \right),
\end{equation*}
where $\gamma$ is the Euler constant, and $\zeta$ is the Riemann zeta function.

\end{definition}

One may further extend the definition to the Gamma class of a vector bundle by defining the Gamma class of the vector bundle to be the product over its Chern roots. The Gamma class is multiplicative with respect to exact sequences of vector bundles.

The Gamma conjecture I states that when the principal asymptotic class is well-defined, it is equal to the Gamma class (up to a non-zero constant multiple).

\begin{conjecture}[Gamma Conjecture I \cite{MR3536989, MR4384381}]

Let $X$ be a Fano manifold such that $c_1(X)\star_0$ has a simple rightmost eigenvalue (Definition \ref{defRightmost}). Then the Gamma class of $X$ (Definition \ref{defGammaClass}) gives the principal asymptotic class of $X$ (Definition \ref{defPrinClass}), i.e.,
\begin{equation*}
\hat{\Gamma}_X \in \mathbb{C} \cdot A_X.
\end{equation*}

\end{conjecture}

Note that the principal asymptotic class is defined and computed using information depending on the symplectic structure of the Fano manifold, e.g., the operator $c_1(X)\star_0$, and the $J$-function. If the Gamma conjecture I is true, then the $1$-dimensional subspace spanned by the principal asymptotic class, $\mathbb{C} \cdot A_X$, is in fact a topological invariant.


\subsection{Riemann-Liouville Integral}

The Riemann-Liouville integral, denoted as $I_\alpha \langle f \rangle$, extends the concept of the $n$-th repeated integral to cases where $n$ is a complex number $\alpha$.
In this work, we further generalize this concept to cohomology classes $\alpha$. For a cohomology class $\alpha$, let $\alpha_0$ represents its $H^0$ component.

\begin{definition}[Riemann-Liouville Integral]

Given a cohomology-valued continuous function $f\in C^0([0,\infty)) \otimes H^*(X)$ and a cohomology class $\alpha \in H^*(X)$ with $\Re \alpha_0 >0$, the \boldemph{Riemann-Liouville integral} of $f$ with respect to $\alpha$ is defined as:
\begin{equation*}
I_\alpha \langle f \rangle (t) \defeq \int_0^t f(x) \; \frac{(t-x)^{\alpha-1}}{\Gamma(\alpha)} \; dx.
\end{equation*}
Additionally, we define $I_0 \langle f \rangle (t) = f(t)$.

\end{definition}

We introduce a norm on $H^*(X)$ defined as the operator norm. Specifically, we first fix an arbitrary norm on $H^*(X)$. Then, for $\alpha\in H^*(X)$, we define $\norm{\alpha}$ by:
\begin{equation*}
\norm{\alpha} \defeq \norm{\alpha \cup - : H^*(X) \rightarrow H^*(X)}_{\text{op}}.
\end{equation*}
The following proposition establishes some fundamental properties of the Riemann-Liouville integral:

\begin{proposition}\label{AsymRL}

Let $\alpha, \beta \in H^*(X)$, with $\Re \alpha_0, \Re \beta_0 >0$. Then, the Riemann-Liouville integrals of a function $f$ satisfy the following properties:
\begin{enumerate}[label=\normalfont (\arabic*)]
\item $I_\alpha \langle I_\beta \langle f \rangle \rangle = I_{\alpha+\beta} \langle f \rangle$,
\item $\sum_{m=0}^\infty I_{rm+\alpha} \langle 1 \rangle = \frac{1}{r} \sum_{k=0}^{r-1} I_\alpha \langle e^{\xi_r^k t} \rangle$, where $\xi_r=e^{\frac{2 \pi \mathtt{i}}{r}}$,
\item for $\Re \lambda>0$, $I_\alpha \langle e^{\lambda t} \rangle (t) = \frac{1}{\lambda^\alpha} e^{\lambda t} - \frac{1}{\lambda \Gamma(\alpha)} t^{\alpha-1} + \mathit{o} \left( \norm*{t^{\alpha-1}} \right)$ as $t\to +\infty$, where the argument of $\lambda$ is taken to be $\arg \lambda\in (-\frac{\pi}{2}, \frac{\pi}{2})$ in $\lambda^\alpha$,
\item for $\Re{\lambda} \leq 0$,  $I_\alpha \langle e^{\lambda t} \rangle (t) = \mathit{o}(t^m)$ as $t\to +\infty$ for some $m\in\mathbb{R}$.
\end{enumerate}

\end{proposition}

\begin{proof}
\; \\
\begin{enumerate}
\item To verify the first property, we compute the left-hand side directly:
\begin{align*}
I_\alpha \langle I_\beta \langle f \rangle \rangle (t) &= \int_0^t \int_0^x f(y) \; \frac{(x-y)^{\beta-1}}{\Gamma(\beta)} \; \frac{(t-x)^{\alpha-1}}{\Gamma(\alpha)} \; dy \; dx \\
&= \int_0^t \int_y^t f(y) \; \frac{(x-y)^{\beta-1}}{\Gamma(\beta)} \; \frac{(t-x)^{\alpha-1}}{\Gamma(\alpha)} \; dx \; dy \\
&= \int_0^t f(y) \; \frac{(t-y)^{\alpha+\beta-1}}{\Gamma(\alpha) \; \Gamma(\beta)} \int_y^t \left( \frac{x-y}{t-y} \right)^{\beta-1} \left( \frac{t-x}{t-y} \right)^{\alpha-1} \hspace{-0.5em} \frac{1}{t-y} \; dx \; dy \\
&= \int_0^t f(y) \; \frac{(t-y)^{\alpha+\beta-1}}{\Gamma(\alpha) \; \Gamma(\beta)} \int_0^1 \left( X \right)^{\beta-1} \; \left( 1-X \right)^{\alpha-1} \; dX \; dy. \\
& \hspace{20em} \left( X=\frac{x-y}{t-y} \right)
\end{align*}
Using the Beta function,
\begin{equation*}
\int_0^1 X^{\beta-1} \; (1-X)^{\alpha-1} \; dX = B(\beta,\alpha) = \frac{\Gamma(\alpha)\;\Gamma(\beta)}{\Gamma(\alpha+\beta)},
\end{equation*} 
we obtain $I_\alpha \langle I_\beta \langle f \rangle \rangle = I_{\alpha + \beta} \langle f \rangle$.

\item
Let $\xi_r=e^{\frac{2 \pi \mathtt{i}}{r}}$ denote the $r$-th root of unity. Then,
\begin{equation*}
\sum_{m=0}^\infty I_{rm+\alpha} \langle 1 \rangle = I_{\alpha} \left\langle \sum_{m=0}^\infty I_{rm} \langle 1 \rangle \right\rangle = I_{\alpha} \left\langle \sum_{m=0}^{\infty} \frac{1}{(rm)!} \; t^{rm} \right\rangle = \frac{1}{r} \sum_{k=0}^{r-1} I_\alpha \langle e^{\xi_r^k t} \rangle.
\end{equation*}

\item
When $\Re{\lambda} > 0$, we have
\begin{align*}
e^{-\lambda t} I_{\alpha}\langle e^{\lambda t} \rangle (t) &= e^{-\lambda t} \int_0^{t} e^{\lambda x} \; \frac{(t-x)^{\alpha-1}}{\Gamma(\alpha)} \; dx \\
&= \int_0^t e^{-\lambda(t-x)} \; \frac{(t-x)^{\alpha-1}}{\Gamma(\alpha)} \; dx \\
&= \int_0^t e^{-\lambda y} \; \frac{y^{\alpha-1}}{\Gamma(\alpha)} \; dy \to \frac{1}{\lambda^{\alpha}} \text{ as } t \to +\infty. \quad (y=t-x)
\end{align*}
Thus,
\begin{equation*}
I_{\alpha} \langle e^{\lambda t} \rangle = \frac{1}{\lambda^{\alpha}} \; e^{\lambda t} \cdot \left( 1+\mathit{o}(1) \right) \text{ as } t\to +\infty.
\end{equation*}
Applying integration by parts iteratively, we obtain the series:
\begin{align*}
I_{\alpha} \langle e^{\lambda t} \rangle (t) =& \int_0^t e^{\lambda x} \; \frac{(t-x)^{\alpha-1}}{\Gamma(\alpha)} \; dx \\
=& \frac{1}{\lambda^{\alpha}} \; e^{\lambda t} - e^{\lambda t} \int_t^\infty \hspace{-0.5em} e^{-\lambda x} \; \frac{x^{\alpha-1}}{\Gamma(\alpha)} \; dx \\
=& \frac{1}{\lambda^{\alpha}} \; e^{\lambda t} - e^{\lambda t} \left( \frac{1}{\lambda} \; e^{-\lambda t} \; \frac{t^{\alpha-1}}{\Gamma(\alpha)} + \int_t^\infty \frac{1}{\lambda} \; e^{-\lambda x} \; \frac{x^{\alpha-2}}{\Gamma(\alpha-1)} \; dx \right) \\
=& \frac{1}{\lambda^{\alpha}} \; e^{\lambda t} - \sum_{\ell=1}^{L} \frac{1}{\lambda^\ell} \; \frac{t^{\alpha-\ell}}{\Gamma(1+\alpha-\ell)} - \int_t^\infty \frac{1}{\lambda^L} \; e^{-\lambda (x-t)} \; \frac{x^{\alpha-(L+1)}}{\Gamma(\alpha-L)} \; dx.
\end{align*}
For $L$ such that $L> \Re{\alpha_0}-1$,
\begin{align*}
 & \norm*{t^{-(\alpha-L)} \int_t^\infty \frac{1}{\lambda^L} \; e^{-\lambda (x-t)} \; \frac{x^{\alpha-(L+1)}}{\Gamma(\alpha-L)} \; dx} \\
\leq& \int_t^\infty \hspace{-0.5em} e^{-(\Re{\lambda})(x-t)} \; \norm*{ \; t^{-(\alpha-L)} \; \frac{x^{\alpha-(L+1)}}{\Gamma(\alpha-L)} \; } \; dx \\
\leq& \int_1^\infty \hspace{-0.5em} e^{-t(\Re{\lambda})(y-1)} \; \norm*{ \; \frac{y^{\alpha-(L+1)}}{\Gamma(\alpha-L)} \; } \; dy \to 0 \text{ as } t \to +\infty. \quad (y=x/t)
\end{align*}

Therefore, We have the asymptotic series as $t\to +\infty$, for sufficiently large $L$,
\begin{equation*}
I_\alpha \langle e^{\lambda t} \rangle (t) = \frac{1}{\lambda^\alpha} \; e^{\lambda t} - \left( \sum_{\ell=1}^L \frac{1}{\Gamma(\alpha-\ell+1) \lambda^\ell} \; t^{\alpha-\ell} + \mathit{o} \left( \norm*{t^{\alpha-L}} \right) \right).
\end{equation*}
In particular, this implies the conclusion 
\begin{equation*}
I_\alpha \langle e^{\lambda t} \rangle (t) = \frac{1}{\lambda^\alpha} \; e^{\lambda t} - \frac{1}{\lambda \; \Gamma(\alpha)} \; t^{\alpha-1} + \mathit{o} \left( \norm*{t^{\alpha-1}} \right).
\end{equation*}

\item
When $\Re{\lambda} \leq 0$, we have
\begin{align*}
\norm*{ \; I_{\alpha} \langle e^{\lambda t} \rangle (t) \; } =& \norm*{ \; \int_0^t e^{\lambda x} \; \frac{(t-x)^{\alpha-1}}{\Gamma(\alpha)} \; dx \; } \\
\leq& \int_0^t e^{(\Re{\lambda}) x} \; \norm*{ \; \frac{(t-x)^{\alpha-1}}{\Gamma(\alpha)} \; } \; dx \\
\leq& \int_0^t \norm*{ \; \frac{(t-x)^{\alpha-1}}{\Gamma(\alpha)} \; } \; dx = \mathit{o}(t^m) \text{ as } t\to +\infty \text{ for some } m\in\mathbb{R}. \qedhere
\end{align*}
\end{enumerate}
\end{proof}
}


\section{Asymptotically Mittag-Leffler}\label{sec:DisGammaI}
{

This section explores the relationship between $t^{\frac{1}{2} \dim}J(t)$ and the Mittag-Leffler function, a generalization of the exponential function defined by: 
\begin{equation*}
E_{\alpha, \beta}(z) = \sum_{m=0}^\infty \frac{1}{\Gamma \left( \alpha m+\beta \right)} \; z^n.
\end{equation*}
For positive real $\alpha$, a suitably modified Mittag-Leffler function exhibits the following asymptotic behaviour as $t\to +\infty$ {\cite[Chapter~4]{MR4179587}}:
\begin{equation*}
t^{\beta} E_{\alpha, 1+\beta}(t^\alpha) = \sum_{m=0}^\infty \frac{1}{\Gamma \left( 1+\alpha m+\beta \right)} \; t^{\alpha m + \beta} = \frac{1}{\alpha} \; e^{t} \cdot (1+\mathit{o}(1)).
\end{equation*}
On the other hand, Proposition \ref{propContAsym} establish that, under the condition that the linear operator $c_1\star_0$ has a simple rightmost eigenvalue, $t^{\frac{1}{2}\dim} J(t)$ displays a similar asymptotic form as $t\to +\infty$:
\begin{equation*}
t^{\frac{1}{2} \dim} J(t) = \sum_{m=0}^\infty J_{rm} \; t^{rm+\frac{1}{2}\dim+c_1} =  e^{Tt} \cdot \left( A + \mathit{o}(1) \right). \tag{$\star$} \label{eq:eqContAsymJ}
\end{equation*}
Motivated by the resemblance of these asymptotic expressions, we propose that for a broad class of Fano manifolds, the coefficients $J_{rm}$ of the series expansion of $t^{\frac{1}{2}\dim} J(t)$ satisfy the following asymptotic equivalance as $m\to \infty$:
\begin{equation*}
J_{rm} = \frac{T^{rm+\frac{1}{2} \dim X+c_1(X)}}{\Gamma \left( 1+rm+\frac{1}{2} \dim X+c_1(X) \right)} \cdot \left( A + \mathit{o}(1) \right). \tag{$\star\star$} \label{eq:eqDisAsymJ}
\end{equation*}
We prove that condition $\eqref{eq:eqDisAsymJ}$ implies condition $\eqref{eq:eqContAsymJ}$. Consquently, when $c_1 \star_0$ has a simple rightmost eigenvalue, the factor $A$ in $\eqref{eq:eqContAsymJ}$ is automatically a principal asymptotic class. This provides an alternative method for determining the principal asymptotic class $A$ for Fano manifold. Finally, we illustrate the computation of the principal asymptotic class for projective spaces.


\subsection{Definition}

We define a property for general cohomology-valued power series, based on their coefficients, that captures the similarity between these coefficients and those of the Mittag-Leffler function.

\begin{definition}[Asymptotically Mittag-Leffler]\label{defaML}

Let $X$ be a smooth manifold. Consider a cohomology-valued power series in $t$ of the form:
\begin{equation*}
\alpha (t) = \sum_{m=0}^\infty \alpha_{rm} t^{rm+\beta}
\end{equation*}
where $\alpha_{rm}, \beta \in H^*(X)$, and $r\in\mathbb{N}_{>0}$. 
We say that $\alpha(t)$ is \boldemph{$(T,\theta,A)$-scaled asymptotically Mittag-Leffler} ($(T,\theta,A)$-aML) for $T \in \mathbb{R}_{>0}$, $\theta\in\mathbb{R}$ and $A \in H^*(X)$ if the coefficients $\alpha_{rm}$ satisfy the asymptotic equivalence as $m \to \infty$:
\begin{equation*}\label{TotalDisAsymp}
\alpha_{rm} = \frac{T^{rm+\beta}e^{\mathtt{i}\theta(rm+\beta)}}{\Gamma \left( 1+rm+\beta \right)} \cdot \left(A + \mathit{o}(1) \right).
\end{equation*}
When $\theta=0$, we may omit $\theta$ and say a series is $(T,A)$-scaled asymptotically Mittag-Leffler.

\end{definition}

If $\alpha(t)$ is $(T,\theta,A)$-scaled asymptotically Mittag-Leffler, then its leading term resembles a modified Mittag-Leffler function:
\begin{equation*}
\sum_{m=0}^\infty \frac{1}{\Gamma(1+rm+\beta)} (Te^{\mathtt{i}\theta} t)^{rm+\beta} A=(Te^{\mathtt{i}\theta}t)^{\beta} \; E_{r,1+\beta} \left( (Te^{\mathtt{i}\theta}t)^r \right) \; A.
\end{equation*}
We can further refine this condition by considering an asymptotic expansion with more terms.  Instead of simply $A + \mathit{o}(1)$, wa can examine an expansion of the form $A_0 + A_1 m^{-1} + \ldots + A_L m^{-L} + \mathit{o}(m^{-L})$ as $m\to \infty$, i.e.,
\begin{equation*}
\alpha_{rm} = \frac{T^{rm+\beta}e^{\mathtt{i}\theta(rm+\beta)}}{\Gamma \left( 1+rm+\beta \right)} \cdot \left(\sum_{\ell=0}^L A_\ell m^{-\ell} + \mathit{o}(m^{-L}) \right).
\end{equation*}
Direct verification reveals the following fundamental properties of asymptotically Mittag-Leffler series:

\begin{proposition}[Basic Properties of Asymptotically Mittag-Leffler Series]\label{propAML}

Let 
\begin{equation*}
\alpha(t) = \sum_{m=0}^\infty \alpha_{rm} \: t^{rm+\beta}
\end{equation*}
be a cohomology-valued series that is $(T, \theta, A)$-scaled asymptotically Mittag-Leffler ($(T,\theta,A)$-aML). Then:
\begin{enumerate}[label=\normalfont \arabic*.]
\item $\alpha(t)$ is also $(T,\theta+2\pi \frac{k}{r},e^{-2\pi\frac{k}{r}\beta} A)$-aML, for all integers $k\in\mathbb{Z}$, in particular, $(T, \theta, A)$-aML is equivalent to $(T, \theta+2\pi, e^{-2\pi\beta} A)$-aML (branch shift), 
\item if $\alpha(t)$ is also $(T',\theta',A')$-aML, $A \neq 0$, and $A' \neq 0$, then $T'=T$, $\theta'=\theta+2\pi\frac{k}{r}$, and $A'= e^{-2\pi \frac{k}{r} \beta} A$ for some $k\in\mathbb{Z}$,
\item $\alpha'(t)$ is $(T, \theta, Te^{\mathtt{i}\theta}A)$-aML,
\item the tail $\sum_{m=M}^\infty \alpha_{rm} \: t^{rm+\beta}=\sum_{m=0}^\infty \alpha_{r(m+M)} \: t^{rm+(rM+\beta)}$ is $(T, \theta, A)$-aML.
\end{enumerate}
However, $t \cdot \alpha(t)$ is not asymptotically Mittag-Leffler for any scaling.

\end{proposition}


\subsection{Properties}

We now prove a general result concerning the asymptotic behaviour of an asymptotically Mittag-Leffler series $\alpha(t)$ as $t\to e^{\mathtt{i}\phi} \cdot \infty$ for some $\phi\in\mathbb{R}$. 

\begin{theorem}\label{thmDisTotalErrortoContTotalError}

Let $X$ be a smooth manifold. Consider a cohomology-valued power series in $t$ of the form:
\begin{equation*}
\alpha (t) = \sum_{m=0}^\infty \alpha_{rm} t^{rm+\beta},
\end{equation*}
where $\alpha_{rm}, \beta \in H^*(X)$, and $r\in\mathbb{N}_{>0}$. Assume $\alpha(t)$ is $(T, \theta, A)$-scaled asymptotically Mittag-Leffler for some $T\in\mathbb{R}_{>0}$, $\theta\in\mathbb{R}$ and $A\in H^*(X)$ (Definition \ref{defaML}). Let $\phi\in\mathbb{R}$. Then $\alpha (t)$ exhibits the following asymptotic equivalence as $t \to e^{\mathtt{i}\phi} \cdot \infty$, 
\begin{equation*}
\alpha (t) = 
\begin{dcases*} 
\frac{1}{r} \; e^{\mathtt{i}(\theta+\phi)\beta} \; e^{T\abs{t}} \cdot \left( A + \mathit{o}(1) \right) & when $\theta+\phi \in \frac{2\pi}{r}\mathbb{Z}$, \\
e^{T\abs{t}} \cdot \mathit{o}(1) & otherwise. 
 \end{dcases*}
\end{equation*}
Moreover, if the coefficients $\alpha_{rm}$ admit the asymptotic series as $m\to\infty$:
\begin{equation*}
\alpha_{rm}=\frac{T^{rm+\beta} e^{\mathtt{i}\theta(rm+\beta)}}{\Gamma(1+rm+\beta)} \cdot \left( \sum_{\ell=0}^L A_\ell m^{-\ell} + \mathit{o}(m^{-L}) \right)
\end{equation*}
for some $A_0(=A), \ldots, A_L \in H^*(X)$, and $L\in\mathbb{N}$, then $\alpha (t)$ possesses the following asymptotic series as $t\to e^{\mathtt{i}\phi} \cdot \infty$:
\begin{equation*}
\alpha (t) = 
\begin{dcases*} 
e^{T\abs{t}} \cdot \left( \sum_{\ell=0}^L \frac{1}{r} \; e^{\mathtt{i}(\theta+\phi)\beta} \; B_\ell \; (T\abs{t})^{-\ell} + \mathit{o} \left( \abs{t}^{-L} \right) \right) & when $\theta+\phi \in \frac{2\pi}{r}\mathbb{Z}$, \\
e^{T\abs{t}} \cdot \mathit{o} \left( \abs{t}^{-L} \right) & otherwise, 
 \end{dcases*}
\end{equation*}
for some $B_0(=A), \ldots, B_L \in H^*(X)$.

\end{theorem}

\begin{proof}
We first consider the case when $\Re(\beta_0)>0$, that is, when the $H^0$ component of $\beta$ has a positive real part. We then extended the result to general $\beta\in H^*(X)$.

By Proposition \ref{propGammaFunction} (3), 
\begin{equation*}
\frac{1}{\Gamma(1+rm+\beta)}m^{-\ell} = \frac{1}{\Gamma(1+rm+\beta+\ell)} \cdot \left( r^\ell + \mathit{O}(m^{-1}) \right). 
\end{equation*}
Thus the asymptotic series of $\alpha_{rm}$ can be rewritten as
\begin{align*}
\alpha_{rm} &= \frac{T^{rm+\beta} e^{\mathtt{i}\theta(rm+\beta)}}{\Gamma(1+rm+\beta)} \cdot \left( \sum_{\ell =0}^L A_\ell m^{-\ell} + \mathit{o}(m^{-L}) \right) \\
&= \sum_{\ell =0}^{L} \frac{T^{rm+\beta} e^{\mathtt{i}\theta(rm+\beta)} B_\ell}{\Gamma(1+rm+\beta+\ell)} + \mathit{o} \left( \frac{T^{rm+\beta} e^{\mathtt{i}\theta(rm+\beta)}}{\Gamma(1+rm+\beta+L)} \right).
\end{align*}
Note that $B_0 = A_0 = A$. In other words, as $m \to \infty$, we have
\begin{equation*}
\alpha_{rm} = \sum_{\ell =0}^L \frac{T^{rm+\beta} e^{\mathtt{i}\theta(rm+\beta)} B_\ell}{\Gamma(1+rm+\beta+\ell)} + \frac{T^{rm+\beta} e^{\mathtt{i}\theta(rm+\beta)} R_L(m)}{\Gamma(1+rm+\beta+L)}
\end{equation*}
for some $L\in \mathbb{N}$, $R_L(m)\in H^*(X)$, with $\lim_{m\to \infty} R_L(m)=0$. Let $\tau=T\abs{t}$. Then $\alpha(t)$ can be expressed as 
\begin{equation*}
\alpha(t) = \underbrace{\sum_{m=0}^\infty \sum_{\ell =0}^L \frac{e^{\mathtt{i}(\theta+\phi)(rm+\beta)} B_\ell}{\Gamma(1+rm+\beta+\ell)} \tau^{rm+\beta}}_{\defeq L(\tau)} + \underbrace{\sum_{m=0}^\infty \frac{e^{\mathtt{i}(\theta+\phi)(rm+\beta)} R_L(m)}{\Gamma(1+rm+\beta+L)} \tau^{rm+\beta}}_{\defeq R(\tau)}.
\end{equation*}

The leading term of $\alpha(t)$ is given by
\begin{align*}
L(\tau) =& \sum_{m=0}^\infty \sum_{\ell =0}^L \frac{e^{\mathtt{i}(\theta+\phi)(rm+\beta)} B_\ell}{\Gamma(1+rm+\beta+\ell)} \tau^{rm+\beta} \\
=& \sum_{\ell =0}^L B_\ell \; \tau^{-\ell} \sum_{m=0}^{\infty} \frac{e^{\mathtt{i}(\theta+\phi)(rm+\beta)}}{\Gamma(1+rm+\beta+\ell)} \tau^{rm+\beta+\ell} \\
=& \sum_{\ell =0}^L e^{\mathtt{i}(\theta+\phi)\beta} \; B_\ell \; \tau^{-\ell} \sum_{m=0}^{\infty} e^{\mathtt{i}(\theta+\phi)rm} I_{rm+\beta+\ell} \langle 1 \rangle (\tau) \\
=& \sum_{\ell =0}^L e^{\mathtt{i}(\theta+\phi)\beta} \; B_\ell \; \tau^{-\ell} \; I_{\beta+\ell} \left\langle \sum_{m=0}^{\infty} \frac{e^{\mathtt{i}(\theta+\phi)rm}}{(rm)!}\tau^{rm} \right\rangle (\tau) \\
=& \sum_{\ell =0}^L e^{\mathtt{i}(\theta+\phi)\beta} \; B_\ell \; \tau^{-\ell} \; \frac{1}{r} \; \sum_{k=0}^{r-1} I_{\beta+\ell} \left\langle e^{\xi_{r}^k e^{\mathtt{i}(\theta+\phi)}\tau} \right\rangle (\tau).
\end{align*}
By Proposition \ref{AsymRL}, if $\Re(\xi_{r}^k e^{\mathtt{i}(\theta+\phi)}) \leq 0$ for all $k\in\mathbb{Z}$, then as $\tau \to +\infty$,
\begin{equation*}
\sum_{k=0}^{r-1} I_{\beta+\ell} \left\langle e^{\xi_{r}^k e^{\mathtt{i}(\theta+\phi)}\tau} \right\rangle (\tau) = \mathit{o}(\tau^m)
\end{equation*}
for some real number $m$.
If there exists $k\in\mathbb{Z}$ such that $\Re(\xi_{r}^k e^{\mathtt{i}(\theta+\phi)}) > 0$, let $k_0\in\mathbb{Z}$ be such that $\Re(\xi_{r}^k e^{\mathtt{i}(\theta+\phi)})$ attains  maximum among all $k\in\mathbb{Z}$, and $\frac{2\pi}{r}k_0+\theta+\phi \in (-\frac{\pi}{2},\frac{\pi}{2})$. Then as $\tau \to +\infty$, 
\begin{equation*}
\sum_{k=0}^{r-1} I_{\beta+\ell} \left\langle e^{\xi_{r}^k e^{\mathtt{i}(\theta+\phi)}\tau} \right\rangle (\tau) = (\xi_{r}^{k_0} e^{\mathtt{i}(\theta+\phi)})^{-(\beta+\ell)} \; e^{\xi_{r}^{k_0} e^{\mathtt{i}(\theta+\phi)}\tau} \cdot \left( 1+\mathit{o}(e^{-c \tau}) \right),
\end{equation*}
where $\xi_r=e^{\frac{2 \pi \mathtt{i}}{r}}$, and $c>0$ is some constant depending on $\theta+\phi$, independent of $\ell$.
Therefore, in particular, we have the following asymptotics of $L(\tau)$ as $\tau\to+\infty$:
\begin{align*}
L(\tau) &= \sum_{\ell =0}^L e^{\mathtt{i}(\theta+\phi)\beta} \; B_\ell \; \tau^{-\ell} \; \frac{1}{r} \; \sum_{k=0}^{r-1} I_{\beta+\ell} \left\langle e^{\xi_{r}^k e^{\mathtt{i}(\theta+\phi)}\tau} \right\rangle (\tau) \\
&=
\begin{dcases*} 
e^{\tau} \cdot \left( \sum_{\ell=0}^L \frac{1}{r} \; e^{\mathtt{i}(\theta+\phi)\beta} \; B_\ell \; (\tau)^{-\ell} + \mathit{o} \left( e^{-c \tau} \right) \right) & when $\theta+\phi \in \frac{2\pi}{r}\mathbb{Z}$, \\
e^{\tau} \cdot \mathit{o} \left( e^{-c \tau} \right) & otherwise. 
 \end{dcases*}
\end{align*}
for some constant $c>0$ depending on $\theta+\phi$.

We now estimate the remaining term $R(\tau)$. Set 
\begin{equation*}
R'_{m}=e^{\mathtt{i}(\theta+\phi)(rm+\beta)} R_L(m), \quad \rho(\tau)=\sum_{m=0}^\infty \frac{R'_L(m)}{(rm)!} (\tau)^{rm}.
\end{equation*}
Since $\lim_{m\to \infty}R'_L(m)=\lim_{m\to \infty} e^{\mathtt{i}(\theta+\phi)\beta} R_L(m)=0$, $\rho(t)$ is an entire function. We can then express the remaining term $R(t)$ using $\rho(t)$ as: 
\begin{equation*}
R(\tau) = \sum_{m=0}^\infty \frac{e^{\mathtt{i}(\theta+\phi)(rm+\beta)} R_L(m)}{\Gamma(1+rm+\beta+L)} \tau^{rm+\beta} = \tau^{-L} \; I_{\beta+L}\langle \rho \rangle (\tau).
\end{equation*}
For $y \geq 1$, we have
\begin{equation*}
\norm*{y^{\beta}} \leq \sum_{k=0}^{\infty} \frac{1}{k!} \norm*{ (\log y)^k \beta^k } \leq \sum_{k=0}^{\infty} \frac{1}{k!} (\log y)^k \norm*{\beta}^k = y^{\norm*{\beta}}.
\end{equation*}
Since $\lim_{m\to \infty}R'_L(m)=0$, we have $\lim_{\tau\to +\infty} e^{-\tau} \, \rho(\tau)=0$, i.e.,
\begin{equation*}
\forall \varepsilon >0, \; \exists \tau_0 >0 \text{ such that } \; \forall \tau>\tau_0, \text{ we have } \; \norm*{\rho(\tau)} < \varepsilon \: e^{\tau}.
\end{equation*}
We compute
\begin{align*}
 & \norm*{e^{-\tau} \: \Gamma \left( \beta+L \right) \: I_{\beta+L} \langle \rho \rangle (\tau)} \\
=& \norm*{e^{-\tau} \int_0^\tau \rho(x) (\tau-x)^{\beta+L-1} dx } \\
\leq & e^{-\tau} \int_0^\tau \norm*{\rho(x) (\tau-x)^{\beta+L-1}} dx \\
< & e^{-\tau} \int_0^{\tau_0} \norm*{\rho(x) (\tau-x)^{\beta+L-1}} dx + \varepsilon e^{-\tau} \int_{\tau_0}^\tau e^{x} \norm*{(\tau-x)^{\beta+L-1}} dx \\
= & e^{-\tau} \int_0^{\tau_0} \norm*{\rho(x) (\tau-x)^{\beta+L-1}} dx + \varepsilon \int_{0}^{\tau-\tau_0} e^{-y} \norm*{y^{\beta+L-1}} dy \quad (y=\tau-x) \\
= & e^{-\tau} \int_0^{\tau_0} \norm*{\rho(x) (\tau-x)^{\beta+L-1}} dx + \begin{aligned}[t] & \varepsilon \int_{0}^{1} e^{-y} \norm*{y^{\beta+L-1}} dy \\
& + \varepsilon \int_{1}^{\tau-\tau_0} e^{-y} y^{\norm*{\beta}+L-1} dy, \end{aligned}
\end{align*}
which converges to $\varepsilon C$ for some constant $C>0$ independent of $\varepsilon$ as $\tau\to +\infty$.
So we have
\begin{equation*}
\limsup_{\tau\to +\infty} \norm*{e^{-\tau} \: \Gamma \left( \beta+L \right) \: I_{\beta+L} \langle \rho \rangle (\tau)} \leq \varepsilon \: C.
\end{equation*}
Since $\varepsilon>0$ is arbitrary, we then have 
\begin{equation*}
\lim_{\tau\to +\infty} \frac{R(\tau)}{e^{\tau} \tau^{-L}} = \lim_{\tau\to +\infty} e^{-\tau} I_{\beta+L} \langle \rho \rangle (\tau) =0, \quad \text{i.e.,} \quad R(\tau) =  e^{\tau} \cdot \mathit{o}(\tau^{-L}).
\end{equation*}
Combining the leading term $L(\tau)$ with the remaining term $R(\tau)$, we have the desired asymptotics as $\tau\to +\infty$:
\begin{equation*}
L(\tau)+R(\tau) = 
\begin{dcases*} 
e^{\tau} \cdot \left( \sum_{\ell=0}^L \frac{1}{r} \; e^{\mathtt{i}(\theta+\phi)\beta} \; B_\ell \; (\tau)^{-\ell} + \mathit{o} \left( \tau^{-L} \right) \right) & when $\theta+\phi \in \frac{2\pi}{r}\mathbb{Z}$, \\
e^{\tau} \cdot \mathit{o} \left( \tau^{-L} \right) & otherwise. 
 \end{dcases*}
\end{equation*}

For general $\beta\in H^*(X)$, one can express $\alpha(t)$ as:
\begin{equation*}
\alpha(t) = \sum_{m=0}^{M-1} \alpha_{rm} \: t^{rm+\beta} + \sum_{m=0}^\infty \alpha_{r(m+M)} \: t^{rm+(rM+\beta)}
\end{equation*}
such that $rM+\beta$ has a positive real $H^0$ component. As the property of being asymptotically Mittag-Leffler is preserved under taking tails (Proposition \ref{propAML} (4)), the above computation can be applied to the tail part of the sum. The remaining part exhibits polynomial growth as $t\to e^{\mathtt{i}\phi}\cdot\infty$, which is dominated by the exponential growth of the tail part. Thus, the whole series also exhibits exponential growth, as does its tail, i.e., as $t\to e^{\mathtt{i}\phi}\cdot \infty$,
\begin{equation*}
\alpha(t) =
\begin{dcases*} 
e^{T\abs*{t}} \cdot \left( \sum_{\ell=0}^L \frac{1}{r} \; e^{\mathtt{i}(\theta+\phi)\beta} \; B_\ell \; (T\abs*{t})^{-\ell} + \mathit{o} \left( \abs*{t}^{-L} \right) \right) & when $\theta+\phi \in \frac{2\pi}{r}\mathbb{Z}$, \\
e^{T\abs*{t}} \cdot \mathit{o} \left( \abs*{t}^{-L} \right) & otherwise. 
 \end{dcases*}
\end{equation*}
\end{proof}

\begin{example}
Consider a manifold $X$ and fix a scale $(T,\theta,A)$, where $T\in\mathbb{R}_{>0}$, $\theta\in\mathbb{R}$ and $A\in H^*(X)$ is non-zero. For any $\phi\in\mathbb{R}$ and sufficiently small $\varepsilon >0$ satisfying $\theta+\phi\notin\frac{2\pi}{r}\mathbb{Z}$, the series 
\begin{equation*}
\alpha_{\phi,\varepsilon}(t) \defeq \sum_{m=0}^\infty \left( \frac{T^{rm+\beta} e^{\mathtt{i}\theta(rm+\beta)}}{\Gamma (1+rm+\beta)} + \frac{(T-\varepsilon)^{rm+\beta} e^{-\mathtt{i}\phi(rm+\beta)}}{\Gamma (1+rm+\beta)} \right) t^{rm+\beta} A
\end{equation*}
is $(T,\theta,A)$-scaled aymptotically Mittag-Leffler. Furthermore, it  exhibits the asymptotic behaviour $\alpha_{\phi,\varepsilon}(t) = e^{(T-\varepsilon) \abs{t}} \cdot \left( \frac{1}{r}A + \mathit{o}(1) \right)$ as $t\to e^{\mathtt{i}\phi} \cdot \infty$. This example demonstrates that, under the condition $\theta+\phi\notin\frac{2\pi}{r}\mathbb{Z}$, the result in Theorem \ref{thmDisTotalErrortoContTotalError} can be improved to the form $e^{(T-\varepsilon) \abs{t}} \cdot \left( \frac{1}{r}A + \mathit{o}(1) \right)$ as $t\to e^{\mathtt{i}\phi} \cdot \infty$ for specific series. However, the value of $\varepsilon$ is series-dependent and necessitates additional information beyond the property of being $(T,\theta,A)$-scaled asymptotically Mittag-Leffler.
\end{example}

We can now apply the theorem to the case where $X$ is a Fano manifold and $\alpha(t)$ is $t^{\frac{1}{2} \dim X} J(t)$ of $X$, where $J(t)$ is the $J$-function of $X$.

\begin{corollary}\label{corDisAsymtoContAym}

Let $X$ be a Fano manifold. Expand the $J$-function $J(t)=J(c_1(X) \log t, 1)$ as a power series:
\begin{equation*}
J(t) = \sum_{m=0}^{\infty} J_{rm} t^{rm+c_1(X)},
\end{equation*}
where $r$ is the Fano index of $X$ (Definition \ref{defJFunction}). If $t^{\frac{1}{2}\dim X} J(t)$ is $(T,A)$-scaled asymptotically Mittag-Leffler (Definition \ref{defaML}) for some $T\in\mathbb{R}_{>0}$ and $A\in H^*(X)$, then $J(t)$ has the following asymptotic equivalence as $t\to +\infty$:
\begin{equation*}
J(t) = \frac{1}{r} \; t^{-\frac{1}{2}\dim X} \; e^{Tt} \cdot \left( A+\mathit{o}(1) \right).
\end{equation*}

\end{corollary}

\begin{proof}
By setting $\beta = \frac{1}{2} \dim X + c_1(X)$, $\theta=\phi=0$ and $\alpha(t) = t^{\frac{1}{2} \dim X} J(t)$, the result follows directly from Theorem \ref{thmDisTotalErrortoContTotalError}.
\end{proof}

According to Proposition \ref{propContAsym}, the $(T,A)$ pair in the asymptotic expansion as $t\to +\infty$,
\begin{equation*}
J(t) = \frac{1}{r} t^{-\frac{1}{2}\dim X} e^{Tt} \cdot (A+\mathit{o}(1)),
\end{equation*}
is already known when $c_1\star_0$ has a simple rightmost eigenvalue. Therefore, under the assumptions that $c_1 \star_0$ have a simple rightmost eigenvalue and $t^{\frac{1}{2} \dim X} J(t)$ is $(T,A)$-scaled asymptotically Mittag-Leffler, the scaling $(T,A)$ can be identified with the rightmost eigenvalue of $c_1 \star_0$ and principal asymptotic class of $X$, respectively.


Most known examples are $(T,\theta,A)$-scaled asymptotically Mittag-Leffler with $\theta=0$.
Later in this section, we will prove that $t^{\frac{1}{2}\dim X} J(t)$ for projective spaces is $(T,A)$-scaled asymptotically Mittag-Leffler. Furthermore, we will show that the property of the $J$-function being $(T,A)$-scaled asymptotically Mittag-Leffler is preserved under the operations of taking products and hypersurfaces. 

There are also potential examples of $(T,\theta,A)$-scaled asymptotically Mittag-Leffler $J$-function with $e^{\mathtt{i}\theta}\neq 1$. In \cite{counterexamples}, the authors investigated a series of toric Fano manifolds, $X_n \defeq \mathbb{P}_{\mathbb{P}^n} (\mathcal{O} \oplus \mathcal{O}(n))$. They proved that, for odd $n \geq 3$, the unique eigenvalue of $c_1(X_n)\star_0$ with maximum modulus is negative. For example, $\rho(c_1(X_3)\star_0) \approx 26.9877$, and $-\rho(c_1(X_3)\star_0)$ is an eigenvalue of the operator $c_1(X_3)\star_0$.

\begin{figure}[H]
\centering
\includegraphics[width=0.8\textwidth]{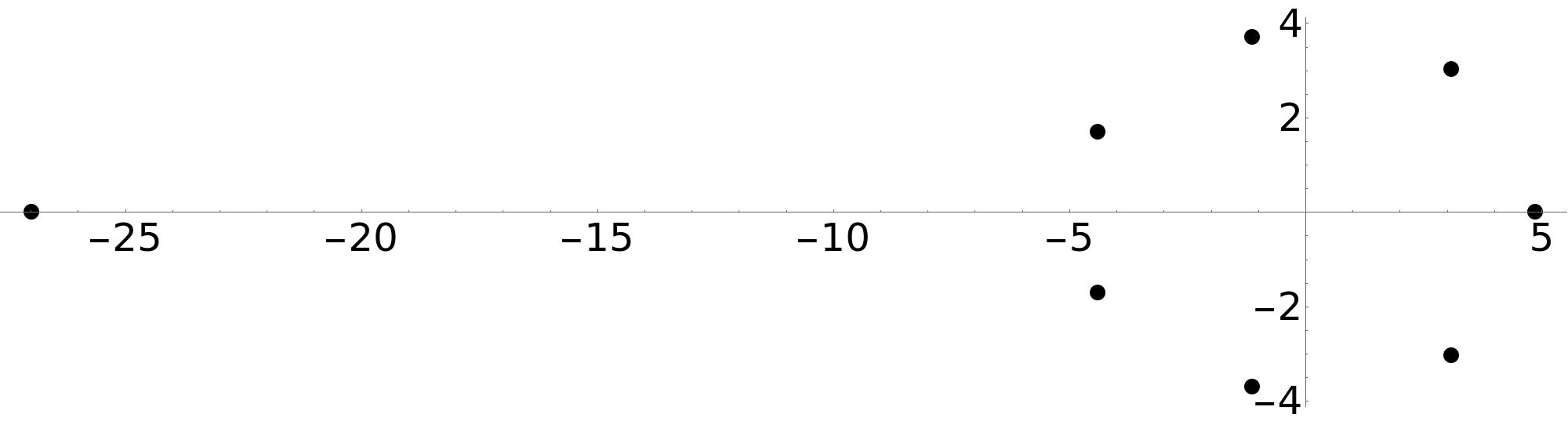}
\caption*{Figure of the distribution of the eigenvalues of $c_1(X_3)\star_0$}
\end{figure}

Let $x_1 \in H^2(X_3)$ be the pullback of the hyperplane class $H$ on the base projective space $\mathbb{P}^3$ and let $x_2 \in H^2(X_3)$ be the Poincaré dual of the divisor $\mathbb{P}_{\mathbb{P}^3}(\mathcal{O}) \subseteq X_3$. Then $-3x_1+x_2$ is Poincaré dual to the divisor $E=\mathbb{P}_{\mathbb{P}^3}(\mathcal{O}(3)) \subseteq X_3$.
The cohomology of $X_3$ is generated by $x_1$ and $x_2$, subject to the relations $x_1^4=0$ and $x_2(-3x_1+x_2)=0$.  
We expect that $X_3$ has an $(T,\theta,A)$-scaled asymptotically Mittag-Leffler $J$-function, where $T=\rho(c_1(X_3)\star_0)$, $\theta=\pi$, and $A\in H^*(X_3)$ is given by:
\begin{align*}
A \propto& \; \hat{\Gamma} \; \Ch{\left( \mathcal{O}_E \otimes \mathcal{O}(-2H) \right)} \\
=& \; \Gamma(1+x_1)^4 \: \Gamma(1-3x_1+x_2) \: \Gamma(1+x_2) \; P(2\pi\mathtt{i}x_1,2\pi\mathtt{i}x_2), 
\end{align*}
where the polynomial $P$ is given by:
\begin{align*}
P(x_1,x_2) =& \left( 1-e^{-(-3x_1+x_2)} \right) e^{-2x_1} \\
=& (-3x_1+x_2) \left( 1-\frac{1}{2}x_1+\frac{1}{2}x_1^2-\frac{5}{24}x_1^3 \right). 
\end{align*}
If the $J$-function of $X_3$ is indeed $(T,\theta,A)$-scaled asymptotically Mittag-Leffler, then the expression in the following table is expected to converge to $A$.
We have computed several values of this expression, represented as a vectors in the basis $(1, x_2, x_1, x_1 x_2, x_1^2, x_1^2 x_2, x_1^3, x_1^3 x_2)$:

\begingroup
\renewcommand{\arraystretch}{1.2}
\begin{table}[H]
\centering
\begin{tabular}{|c|c|} 
 \hline
 $m$ & $J_m \; m! \: m^{2+c_1} \; T^{-(m+2+c_1)} \; e^{-\mathtt{i}\theta (m+2+c_1)}$ \\
 \hline \hline
 $14$ & $\begin{aligned} (& 0., -0.248743, 0.746229, 0.187542+0.781449\mathtt{i}, -0.562627-2.34435\mathtt{i}, \\ & 2.18704-0.589182\mathtt{i}, -6.56111+1.76755\mathtt{i}, -3.97697-4.29992\mathtt{i}) \times 10^{-3} \end{aligned}$ \\ 
 \hline
 $15$ & $\begin{aligned} (& 0., -0.249072, 0.747216, 0.184877+0.782483\mathtt{i}, -0.554631-2.34745\mathtt{i}, \\ & 2.19172-0.580808\mathtt{i}, -6.57515+1.74242\mathtt{i}, -3.95630-4.31122\mathtt{i}) \times 10^{-3} \end{aligned}$ \\
 \hline
 $16$ & $\begin{aligned} (& 0., -0.249367, 0.748102, 0.182541+0.783411\mathtt{i}, -0.547422-2.35023\mathtt{i}, \\ & 2.19584-0.573469\mathtt{i}, -6.58753+1.72041\mathtt{i}, -3.93822-4.32113\mathtt{i}) \times 10^{-3} \end{aligned}$ \\
 \hline
 $17$ & $\begin{aligned} (& 0., -0.249635, 0.748904, 0.180477+0.784251\mathtt{i}, -0.541431-2.35275\mathtt{i}, \\ & 2.19952-0.566985\mathtt{i}, -6.59856+1.70096\mathtt{i}, -3.92230-4.32991\mathtt{i}) \times 10^{-3} \end{aligned}$ \\
 \hline
 $30$ & $\begin{aligned} (& 0., -0.251845, 0.755534, 0.166211+0.791193\mathtt{i}, -0.498632-2.37358\mathtt{i}, \\ & 2.22748-0.522166\mathtt{i}, -6.68243+1.56650\mathtt{i}, -3.81519-4.39490\mathtt{i}) \times 10^{-3} \end{aligned}$ \\ 
 \hline
 \hline
 $A$ & $\begin{aligned} (& 0., -0.252094, 0.756281, 0.145512+0.791976\mathtt{i}, -0.436537-2.37593\mathtt{i}, \\ & 2.23873-0.457141\mathtt{i}, -6.71619+1.37142\mathtt{i}, -3.63163-4.42768\mathtt{i}) \times 10^{-3} \end{aligned}$ \\
 \hline 
\end{tabular}
\caption*{Table of some values of $J_m \; m! \: m^{2+c_1} \; T^{-(m+2+c_1)} \; e^{-\mathtt{i}\theta (m+2+c_1)}$ of $\mathbb{P}_{\mathbb{P}^3} (\mathcal{O} \oplus \mathcal{O}(3))$}
\label{table:valueofCoeofX3}
\end{table}
\endgroup

The computation results presented above lend support to our expectation that the $J$-function of $\mathbb{P}_{\mathbb{P}^3} (\mathcal{O} \oplus \mathcal{O}(3))$ is $(T,\theta,A)$-scaled asymptotically Mittag-Leffler, where $T=\rho(c_1(X_3)\star_0)$ and $\theta=\pi$. The observation that $T$ corresponds to the spectral radius of $c_1\star_0$ also holds in other examples where $\theta=0$. Building upon the approach of Proposition \ref{propContAsym}, which connects the simple rightmost eigenvalue and the asymptotic behaviour of the $J$-function, we pose the following question regarding a potential relationship between eigenvalues of $c_1\star_0$ with maximum modulus and the asymptotic behaviour of the coefficients of $J$-function:

\begin{question}\label{questionDiscreteAsym}

Let $X$ be a Fano manifold. For the operator $c_1\star_0 :H^*(X) \rightarrow H^*(X)$, let $T$ be its spectral radius, and let $\theta\in\mathbb{R}$ such that $Te^{\mathtt{i}\theta}$ is an eigenvalue. Is $t^{\frac{1}{2}\dim X}J(t)$ a $(T, \theta, A_\theta)$-scaled asymptotic Mittag-Leffler series for some non-zero $A_\theta \in H^*(X)$?

\end{question}

We now highlight some immediate corollaries that whould follow if the preceding question were answered affirmatively for certain $\theta\in\mathbb{R}$.
\begin{enumerate}
\item If a Fano manifold $X$ satisfies the property stated above for $\theta=0$, and $T$ is a simple eigenvalue, then $T$ is the simple rightmost eigenvalue of $c_1\star_0$. In this case, $A_0$ is the principal asymptotic class.
\item If the property above is satisfied by all eigenvalues with maximum modulus, then the $J$-function would be $(T,\theta,A_\theta)$-scaled asymptotically Mittag-Leffler for multiple values of $\theta$. By Proposition \ref{propAML}, the differences between these $\theta$ values are integer multiples of $\frac{2\pi}{r}$. In this case, all eigenvalues with maximum modulus differ by a multiple of $e^{2\pi\mathtt{i}\frac{k}{r}}$ for some integer $k$.
\end{enumerate}
Note that if a Fano manifold satisfies both conditions 1 and 2 above, then it satisfies Property $\mathcal{O}$ as defined in \cite{MR3536989}.


\subsection{Example: Projective Spaces}

As an example, we verify our claim regarding the asymptotic behaviour of coefficients of the $J$-function fot projective spaces through direct computation.

\begin{proposition}

Let $J(t)=J(c_1 \log t, 1)$ be the $J$-function of $\mathbb{P}^N$. Then $t^{\frac{1}{2} N} J(t)$ is $(T,A)$-scaled asymptotically Mittag-Leffler (Definition \ref{defaML}). Specifically, the coefficients of $t^{\frac{1}{2} N} J(t)$ of $\mathbb{P}^N$ exhibit the asymptotic equivalence as $m\to \infty$:
\begin{equation*}
J_{rm} = \frac{T^{(N+1)m+\frac{1}{2}N+c_1}}{\Gamma \left( 1+(N+1)m+\frac{1}{2}N+c_1 \right)} \cdot \left( A+\mathit{o}(1) \right)
\end{equation*}
where $T=N+1$, and $A=(N+1)^{\frac{1}{2}} (2\pi)^{-\frac{1}{2} N} \Gamma(1+\delta)^{N+1}$, with $\delta\in H^2(\mathbb{P}^N)$ being the hyperplane class. 

\end{proposition}

\begin{proof}
The $J$-function of $X=\mathbb{P}^N$ is given by \cite{MR1354600}:
\begin{equation*}
t^{\frac{1}{2} N}J(t) = \sum_{m=0}^\infty \prod_{k=1}^{m} \frac{1}{(\delta+k)^{N+1}} t^{rm+\frac{1}{2}N+c_1},
\end{equation*}
The coefficients of the $J$-function are given by
\begin{equation*}
J_{rm} = \prod_{k=1}^m \frac{1}{(\delta+k)^{N+1}}=\frac{\Gamma(1+\delta)^{N+1}}{\Gamma(1+m+\delta)^{N+1}}.
\end{equation*}
By the multiplication formula (Proposition \ref{propGammaFunction}(2)), 
\begin{multline*}
\prod_{k=0}^{N} \; \Gamma\left( 1+m+\delta +\frac{k}{N+1} \right) \\
= (2\pi)^{\frac{N}{2}} \: (N+1)^{\frac{1}{2}-(N+1) (1+m+\delta)} \: \Gamma \left( (N+1) (1+m+\delta) \right).
\end{multline*}
Therefore, by Proposition \ref{propGammaFunction}(3),
\begingroup
\allowdisplaybreaks[1]
\begin{align*}
1 &= \lim_{m\to\infty} \frac{\prod_{k=0}^{N} \Gamma\left( 1+m+\delta +\frac{k}{N+1} \right)}{\prod_{k=0}^{N} \Gamma(1+m+\delta) \; (1+m+\delta)^{\frac{k}{N+1}}} \\[3ex]
&= \lim_{m\to\infty} \frac{(2\pi)^{\frac{N}{2}} \; (N+1)^{\frac{1}{2}-(N+1) (1+m+\delta)} \; \Gamma \left( (N+1) (1+m+\delta) \right)}{\Gamma(1+m+\delta)^{N+1} \; (1+m+\delta)^{\frac{1}{2}N}} \\[3ex]
&= \lim_{m\to\infty} \frac{1}{\Gamma(1+m+\delta)^{N+1} \; (1+m+\delta)^{\frac{1}{2}N}} \; \frac{\Gamma \left( N+1+(N+1)m+(N+1)\delta \right)}{(2\pi)^{-\frac{1}{2}N} \; (N+1)^{\frac{1}{2}+N+(N+1)m+(N+1)\delta} } \\[3ex]
&= \lim_{m\to\infty} \frac{1}{\Gamma(1+m+\delta)^{N+1}} \; \frac{\Gamma \left( N+1+(N+1)m+(N+1)\delta -\frac{1}{2}N \right)}{(2\pi)^{-\frac{1}{2}N} \; (N+1)^{\frac{1}{2}+N+(N+1)m+(N+1)\delta-\frac{1}{2}N} } \\[3ex]
&= \lim_{m\to\infty} J_{rm} \; \frac{1}{(N+1)^{\frac{1}{2}} \; (2\pi)^{-\frac{1}{2} N} \; \Gamma(1+\delta)^{N+1}} \; \frac{\Gamma \left( 1+(N+1)m+\frac{1}{2}N+c_1 \right)}{(N+1)^{(N+1)m+\frac{1}{2}N+c_1}} \\[3ex]
&= \lim_{m\to\infty} J_{rm} \; \frac{1}{A} \; \frac{\Gamma \left( 1+(N+1)m+\frac{1}{2}N+c_1 \right)}{T^{(N+1)m+\frac{1}{2}N+c_1}}.
\end{align*}
\endgroup
\end{proof}

Note that for $X=\mathbb{P}^N$, $c_1 \star_0$ has a simple rightmost eigenvalue. Therefore, the principal asymptotic class of $\mathbb{P}^N$ can be given by $A=(N+1)^{\frac{1}{2}} (2\pi)^{-\frac{1}{2} N} \Gamma(1+\delta)^{N+1}$ based on the above computation. As the Gamma class of $\mathbb{P}^N$ is $\hat{\Gamma}_X=\Gamma(1+\delta)^{N+1}$, the above computation also verifies the Gamma conjecture I for projective spaces. 
}


\section{Product}\label{sec:Product}
{
In this section, we aim to prove that the product of two $(T,A)$-scaled asymptotically Mittag-Leffler series maintains the asymptotically Mittag-Leffler property. Subsequently, we will apply this result to the product of two Fano manifolds, each possessing an asymptotically Mittag-Leffler $J$-function.

\begin{proposition}\label{propDisGammaIProduct}

Let $X$ and $Y$ be two smooth manifolds. Let $\alpha^X(t) : \mathbb{R}_{>0} \rightarrow H^*(X)$ and $\alpha^Y(t) : \mathbb{R}_{>0} \rightarrow H^*(Y)$ be two cohomology-valued functions, with the series expansions:
\begin{equation*}
\alpha^X(t) = \sum_{m=0}^\infty \alpha^X_{r_X m} t^{r_X m+\beta_X} \quad \text{and} \quad \alpha^Y(t) = \sum_{m=0}^\infty \alpha^Y_{r_Y m} t^{r_Y m+\beta_Y}, 
\end{equation*}
Suppose that both $\alpha^X(t)$ and $\alpha^Y(t)$ are asymptotically Mittag-Leffler (Definition \ref{defaML}), with scalings $(T_X, A_X)$ and $(T_Y, A_Y)$, respectively. Then, the product $\alpha^X(t) \alpha^Y(t): \mathbb{R}_{>0} \rightarrow H^*(X) \otimes H^*(Y) = H^*(X\times Y)$ is $(T_X+T_Y, \frac{r}{r_X r_Y} A_X \otimes A_Y)$-scaled asymptotically Mittag-Leffler, where $r=\mathrm{GCD}(r_X, r_Y)$.

\end{proposition}

\begin{proof}

\subsubsection*{}

We begin by proving the case where the $H^0$ components of both $\beta_X$ and $\beta_Y$ have a positive real parts. Then, we will extend this result to general $\beta_X$ and $\beta_Y$.
For the sake of simplicity, let us denote
\begin{equation*}
I_\alpha(t) \defeq I_{\alpha} \langle 1 \rangle (t) = \frac{t^{\alpha}}{\Gamma \left( 1+\alpha \right)} \in H^*(X \times Y)
\end{equation*}
where $\alpha \in H^*(X \times Y)$, $t\in \mathbb{R}_{>0}$.

For $\alpha_1\in H^*(X)$ and $\alpha_2\in H^*(Y)$, we write $\alpha_1 \alpha_2$ to represent $\alpha_1 \otimes \alpha_2 \in H^*(X\times Y)$ for brevity.
\begin{align*}
& \alpha^X(t) \alpha^Y(t) \\
&= \left( \sum_{m_X=0}^\infty \alpha^X_{r_X m_X} t^{r_X m_X + \beta_X} \right) \left( \sum_{m_Y=0}^\infty \alpha^Y_{r_Y m_Y} t^{r_Y m_Y + \beta_Y} \right) \\
&= \sum_{n=0}^\infty \Bigg( \sum_{\substack{(m_X, m_Y) \in \mathbb{N}^2 \\ r_X m_X + r_Y m_Y = n}} \hspace{-1em} \alpha^X_{r_X m_X} \alpha^Y_{r_Y m_Y} \Bigg) t^{n+\beta_X+\beta_Y}.
\end{align*}
Therefore, we have
\begin{equation*}
(\alpha^X \alpha^Y)_{n} = \begin{dcases*} \sum_{\substack{(m_X, m_Y) \in \mathbb{N}^2 \\ r_X m_X + r_Y m_Y = n}} \hspace{-1em} \alpha^X_{r_X m_X} \alpha^Y_{r_Y m_Y} & if $r \mid n$, \\
 0 & if $r \nmid n$, \end{dcases*}
\end{equation*}
or, equilvalently,
\begin{equation*}
(\alpha^X \alpha^Y)_{rm} = \hspace{-1em} \sum_{\substack{(m_X, m_Y) \in \mathbb{N}^2 \\ r_X m_X + r_Y m_Y = rm}} \hspace{-1em} \alpha^X_{r_X m_X} \alpha^Y_{r_Y m_Y}.
\end{equation*}
Given that $\alpha^X$ and $\alpha^Y$ are both asymptotically Mittag-Leffler, with scalings $(T_X, A_X)$ and $(T_Y, A_Y)$ respectively, the following asymptotic equivalences hold:
\begin{equation*}
\alpha^X_{r_X m_X} = I_{r_X m_X + \beta_X} (T_X) \cdot \left(A_X + R^X_{r_X m_X} \right),
\end{equation*}
\begin{equation*}
\alpha^Y_{r_Y m_Y} = I_{r_Y m_Y + \beta_Y} (T_Y) \cdot \left(A_Y + R^Y_{r_Y m_Y} \right),
\end{equation*}
for some $R^X_{r_X m_X} \in H^*(X)$, $R^Y_{r_Y m_Y} \in H^*(Y)$ with 
\begin{equation*}
\lim_{m_X \to \infty} R^X_{r_X m_X}=\lim_{m_Y \to \infty} R^Y_{r_Y m_Y}=0.
\end{equation*}


\subsubsection*{(Setup:)}

Therefore, we have
\begingroup
\allowdisplaybreaks[1]
\begin{align*}
 & (\alpha^X \alpha^Y)_{rm} \\[2ex]
=& \hspace{-1em} \sum_{\substack{(m_X, m_Y) \in \mathbb{N}^2 \\ r_X m_X + r_Y m_Y = rm}} \hspace{-1em} I_{r_X m_X + \beta_X}(T_X) \; I_{r_Y m_Y + \beta_Y}(T_Y) \; \left(A_X + R^X_{r_X m_X} \right) \left(A_Y + R^Y_{r_Y m_Y} \right) \\[2ex]
=& \hspace{-1em} \sum_{\substack{(m_X, m_Y) \in \mathbb{N}^2 \\ r_X m_X + r_Y m_Y = rm}} \hspace{-1em} \begin{aligned}[t] & \left(A_X + R^X_{r_X m_X} \right) \left(A_Y + R^Y_{r_Y m_Y} \right) \\
 & \hspace{1em} \cdot \smashoperator{\int\limits_{{[0,T_X]} \times {[0,T_Y]}}} I_{r_X m_X}(t_X) I_{r_Y m_Y}(t_Y) I_{\beta_X-1}(T_X-t_X) I_{\beta_Y-1}(T_Y-t_Y) \: dt_X dt_Y \end{aligned} \\[2ex]
=& \hspace{-1em} \sum_{\substack{(m_X, m_Y) \in \mathbb{N}^2 \\ r_X m_X + r_Y m_Y = rm}} \hspace{-1em} \begin{aligned}[t] & \left(A_X + R^X_{r_X m_X} \right) \left(A_Y + R^Y_{r_Y m_Y} \right) (T_X+T_Y)^{rm+\beta_X+\beta_Y} \\ 
 & \hspace{1em} \cdot \smashoperator{\int\limits_{{[0,S_X]} \times {[0,S_Y]}}} I_{r_X m_X}(s_X) I_{r_Y m_Y}(s_Y) I_{\beta_X-1}(S_X-s_X) I_{\beta_Y-1}(S_Y-s_Y) \: ds_X ds_Y \end{aligned} \tag{$A$} \label{eq:eqA}
\end{align*}
\endgroup
where $(s_X, s_Y) = (t_X, t_Y)/(T_X + T_Y)$ and $(S_X, S_Y) = (T_X, T_Y)/(T_X + T_Y)$.

Let $D, D', D''$ denote the following polytopal regions:
\begin{align*}
  D&=\Poly \left( (-S_Y,S_Y),(S_X,S_Y),(S_X,-S_X) \right), \\
 D'&=\Poly \left( (-S_Y,S_Y),(0,S_Y),(0,0) \right), \\
D''&=\Poly \left( (0,0),(S_X,0),(S_X,-S_X) \right).
\end{align*}
where $\Poly(p_1, p_2, \ldots, p_N)$ represents the convex hull of $p_1, p_2, \ldots, p_N \in \mathbb{R}^2$, i.e.:
\begin{equation*}
\Poly(p_1, p_2, \ldots, p_N) = \left\{ p \;\middle\vert\; p=\sum_{n=1}^N a_n p_n, \text{ where } \sum_{n=1}^N a_n=1 , a_n \geq 0 \right\}.
\end{equation*}
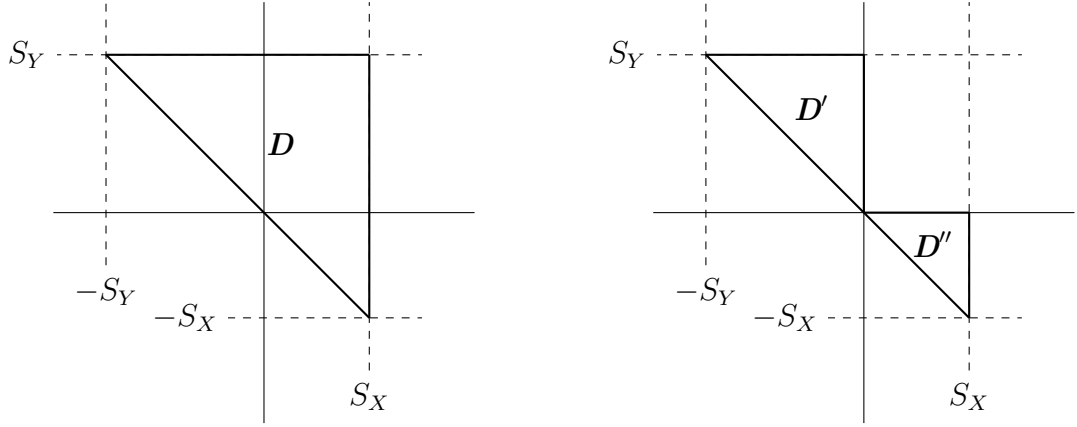
\begin{figure}[h]
\centering
\begin{subfigure}[b]{0.45\textwidth}
\centering
\begin{tikzpicture}[scale=0.7]
\draw (-4,0)--(4,0);
\draw (0,-4)--(0,4);

\draw[thick] (-3,3)--(2,3);
\draw[thick] (2,-2)--(2,3);
\draw[thick] (-3,3)--(2,-2);

\draw[dashed] (-3,-1)--(-3,4);
\draw[dashed] (-1,-2)--(3,-2);
\draw[dashed] (-4,3)--(3,3);
\draw[dashed] (2,-3)--(2,4);

\node[rectangle, fill=white] at (-3, -1-0.5) {$-S_Y$};
\node[rectangle, fill=white] at (-1-0.5, -2) {$-S_X$};
\node[rectangle, fill=white] at (-4-0.5, 3) {$S_Y$};
\node[rectangle, fill=white] at (2, -3-0.5) {$S_X$};

\node at (0.3,1.3) {$\boldemph{D}$};
\end{tikzpicture}
\end{subfigure}
\hfill
\begin{subfigure}[b]{0.45\textwidth}
\centering
\begin{tikzpicture}[scale=0.7]
\draw (-4,0)--(4,0);
\draw (0,-4)--(0,4);

\draw[thick] (-3,3)--(0,3);
\draw[thick] (0,0)--(0,3);
\draw[thick] (2,-2)--(2,0);
\draw[thick] (0,0)--(2,0);
\draw[thick] (-3,3)--(2,-2);

\draw[dashed] (-3,-1)--(-3,4);
\draw[dashed] (-1,-2)--(3,-2);
\draw[dashed] (-4,3)--(3,3);
\draw[dashed] (2,-3)--(2,4);

\node[rectangle, fill=white] at (-3, -1-0.5) {$-S_Y$};
\node[rectangle, fill=white] at (-1-0.5, -2) {$-S_X$};
\node[rectangle, fill=white] at (-4-0.5, 3) {$S_Y$};
\node[rectangle, fill=white] at (2, -3-0.5) {$S_X$};

\node at (-1,2) {$\boldemph{D}'$};
\node at (1.3,-0.6) {$\boldemph{D}''$};
\end{tikzpicture}
\end{subfigure}
\caption*{Figures of the domains $D$, $D'$ and $D''$}
\end{figure}
Utilizing the multivariate beta function, the desired asymptotic can be expressed as an integral over region $D$:
\begin{align*}
L_{rm} & \defeq \int\limits_D I_{rm}(s_X+s_Y) \; I_{\beta_X-1}(S_X-s_X) \; I_{\beta_Y-1}(S_Y-s_Y) ds_X ds_Y \\
&= I_{rm+\beta_X+\beta_Y}(S_X+S_Y) = \frac{1}{\Gamma(1+rm+\beta_X+\beta_Y)}, \tag{$L$} \label{eq:eqL}
\end{align*}
We now proceed to modify $L_{rm}$ step-by-step to obtain our target expression $\eqref{eq:eqA}$.
First, observe that 
\begin{equation*}
D=D' \cup \left( {[0,S_X]} \times {[0,S_Y]} \right) \cup D'' \subseteq \mathbb{R}^2.
\end{equation*}
A term $E^1_{rm}$ is subtracted from $L_{rm}$ to alter the domain of integration from $D$ to ${[0, S_X]} \times {[0, S_Y]}$. The term $E^1_{rm}$ is defined as:
\begin{equation*}
E^1_{rm} \defeq \smashoperator{\int\limits_{D' \cup D''}} I_{rm}(s_X+s_Y) \; I_{\beta_X-1}(S_X-s_X) \; I_{\beta_Y-1}(S_Y-s_Y) ds_X ds_Y.
\end{equation*}
Next, by comparing the coefficients of 
\begingroup
\allowdisplaybreaks[1]
\begin{align*}
& \sum_{m=0}^\infty \sum_{\substack{(m_X, m_Y) \in \mathbb{N}^2 \\ r_X m_X + r_Y m_Y = rm}} \hspace{-1em} I_{r_X m_X}(s_X) I_{r_Y m_Y}(s_Y) \; t^{rm} \\
=& \left( \frac{1}{r_X} \sum_{k=0}^{r_X-1} e^{\xi_{r_X}^k s_X t} \right) \left( \frac{1}{r_Y} \sum_{\ell=0}^{r_Y-1} e^{\xi_{r_Y}^\ell s_Y t} \right) \\
=&\frac{1}{r_X r_Y} \left( \sum_{k=0}^{r_X-1} \frac{1}{r_Y} e^{\left( \xi_{r_X}^k s_X + \xi_{r_Y}^\ell s_Y \right) t} \right) \\
=& \frac{1}{r_X r_Y} \sum_{m=0}^\infty I_{rm}\left( \xi_{r_X}^k s_X + \xi_{r_Y}^\ell s_Y \right) t^{rm} \\
=& \sum_{m=0}^\infty \Biggl( \frac{r}{r_X r_Y} I_{rm}(s_X + s_Y) + \frac{1}{r_X r_Y} \sum_{\substack{0 \leq k < r_X \\ 0 \leq \ell < r_Y \\ \xi_{r_X}^k \neq \xi_{r_Y}^\ell }} I_{rm} \left( \xi_{r_X}^k s_X + \xi_{r_Y}^\ell s_Y \right) \Biggr) t^{rm},
\end{align*}
\endgroup
where $\xi_r^n=e^{2 \pi \mathtt{i}\frac{n}{r}}$, we arrive at the equation
\begin{align*}
& \sum_{\substack{(m_X, m_Y) \in \mathbb{N}^2 \\ r_X m_X + r_Y m_Y = rm}} \hspace{-1em} I_{r_X m_X}(s_X) I_{r_Y m_Y}(s_Y) \\
&= \frac{r}{r_X r_Y} I_{rm}(s_X + s_Y) + \frac{1}{r_X r_Y} \sum_{\substack{0 \leq k < r_X \\ 0 \leq \ell < r_Y \\ \xi_{r_X}^k \neq \xi_{r_Y}^\ell }} I_{rm} \left( \xi_{r_X}^k s_X + \xi_{r_Y}^\ell s_Y \right).
\end{align*}
The preceding equation relates $I_{rm}(s_X + s_Y)$ with $\sum I_{r_X m_X}(s_X) I_{r_Y m_Y}(s_Y)$. Consequently, the corresponding term of modification to $L_{rm}$ is:
\begin{equation*}
E^2_{rm} \defeq \hspace{-1em} \int\limits_{{[0,S_X]} \times {[0,S_Y]}} \sum_{\substack{0 \leq k < r_X \\ 0 \leq \ell < r_Y \\ \xi_{r_X}^k \neq \xi_{r_Y}^\ell }} \begin{aligned}[t] & I_{rm} ( \xi_{r_X}^k s_X+ \xi_{r_Y}^\ell s_Y ) \\
 & \hspace{2em} \cdot I_{\beta_X-1}(S_X-s_X) \; I_{\beta_Y-1}(S_Y-s_Y) ds_X ds_Y. \end{aligned}
\end{equation*}
Finally, by 
\begin{equation*}
A_X A_Y + A_Y R^X_{r_X m_X} + A_X R^Y_{r_Y m_Y} + R^X_{r_X m_X} R^Y_{r_Y m_Y} = \left( A_X + R^X_{r_X m_X} \right) \left( A_Y + R^Y_{r_Y m_Y} \right),
\end{equation*}
the final term to add to $L_{rm}$ is:
\begin{equation*}
E^3_{rm} \defeq \hspace{-2em} \sum_{\substack{(m_X, m_Y) \in \mathbb{N}^2 \\ r_X m_X + r_Y m_Y = rm}} \hspace{-1em} \begin{aligned}[t] & \left( A_Y R^X_{r_X m_X} + A_X R^Y_{r_Y m_Y} + R^X_{r_X m_X} R^Y_{r_Y m_Y} \right) \\
 & \hspace{1em} \cdot \smashoperator{\int\limits_{{[0,S_X]} \times {[0,S_Y]}}} I_{r_X m_X}(s_X) I_{r_Y m_Y}(s_Y) I_{\beta_X-1}(S_X-s_X) I_{\beta_Y-1}(S_Y-s_Y) \: ds_X ds_Y. \end{aligned}
\end{equation*}
Combining the three terms above, we construct our expression $\eqref{eq:eqA}$ using the formula:
\begin{equation*}
\frac{(\alpha^X \alpha^Y)_{rm}}{(T_X+T_Y)^{rm+\beta_X+\beta_Y}} = \frac{r A_X A_Y}{r_X r_Y} L_{rm} - \frac{r A_X A_Y}{r_X r_Y} E^1_{rm} + \frac{A_X A_Y}{r_X r_Y} E^2_{rm} + E^3_{rm}.
\end{equation*}
We claim that:
\begin{align*}
\lim_{m\to\infty} \frac{E^1_{rm}}{L_{rm}} &= 0, \tag{$E1$} \label{eq:eqE1} \\[1ex]
\lim_{m\to\infty} \frac{E^2_{rm}}{L_{rm}} &= 0, \tag{$E2$} \label{eq:eqE2} \\[1ex]
\lim_{m\to\infty} \frac{E^3_{rm}}{L_{rm}} &= 0. \tag{$E3$} \label{eq:eqE3}
\end{align*}


\subsubsection*{(Checking of $\eqref{eq:eqE1}$:)}

We will first verify that $\lim_{m\to\infty} E^1_{rm}/L_{rm}=0$. Recall that 
\begin{align*}
\frac{E^1_{rm}}{L_{rm}} =& \Gamma \left( 1+rm+\beta_X+\beta_Y \right) \\[-1ex]
& \hspace{1em} \cdot \smashoperator{\int\limits_{D' \cup D''}} I_{rm}(s_X+s_Y) \; I_{\beta_X-1}(S_X-s_X) \; I_{\beta_Y-1}(S_Y-s_Y) ds_X ds_Y. 
\end{align*}
We begin by estimating the asmyptotic bahaviour of the integral $E^1_{rm}$ over the region $D'$:
\begingroup
\allowdisplaybreaks[1]
\begin{align*}
 & \Biggl\lVert \begin{aligned}[t] & \Gamma \left( 1+rm+\beta_Y \right) (rm)^{\beta_X} \\
 & \hspace{1em} \cdot \int\limits_{D'} I_{rm}(s_X+s_Y) \: I_{\beta_X-1}(S_X-s_X) \: I_{\beta_Y-1}(S_Y-s_Y) \; ds_X ds_Y \Biggr\rVert \end{aligned} \\[1ex]
=& \Biggl\lVert \begin{aligned}[t] & \Gamma \left( 1+rm+\beta_Y \right) (rm)^{\beta_X} \\
 & \hspace{1em} \cdot \int\limits_{-S_Y}^0 \int\limits_{-s_X}^{S_Y} I_{rm}(s_X+s_Y) \: I_{\beta_X-1}(S_X-s_X) \: I_{\beta_Y-1}(S_Y-s_Y) \; ds_Y ds_X \Biggr\rVert \end{aligned} \\[1ex]
=& \Biggl\lVert \Gamma \left( 1+rm+\beta_Y \right) (rm)^{\beta_X} \int\limits_{-S_Y}^0 I_{rm+\beta_Y}(s_X+S_Y) \: I_{\beta_X-1}(S_X-s_X) \; ds_X \Biggr\rVert \\[1ex]
=& \Biggl\lVert (rm)^{\beta_X} \int\limits_{-S_Y}^0 (s_X+S_Y)^{rm+\beta_Y} \: I_{\beta_X-1}(S_X-s_X) \; ds_X \Biggr\rVert \\[1ex]
\leq& \underbrace{S_Y^{rm}}_{\substack{\text{exponential}\\ \text{as }m\to\infty}} \biggl\lVert \underbrace{(rm)^{\beta_X}}_{\substack{\text{polynomial}\\ \text{as }m\to\infty}} \int_{-S_Y}^0 (s_X+S_Y)^{\beta_Y} \: I_{\beta_X-1}(S_X-s_X) ds_X \biggr\rVert.
\end{align*}
\endgroup
Since $S_Y=T_Y/(T_X+T_Y)<1$, we have 
\begin{align*}
\lim_{m\to\infty} & \Gamma \left( 1+rm+\beta_Y \right) (rm)^{\beta_X} \\[-1ex]
& \hspace{1em} \cdot \int\limits_{D'} I_{rm}(s_X+s_Y) \: I_{\beta_X-1}(S_X-s_X) \: I_{\beta_Y-1}(S_Y-s_Y) \; ds_X ds_Y =0. 
\end{align*}
Because
\begin{equation*}
\lim_{m\to\infty} \frac{\Gamma \left( 1+rm+\beta_X+\beta_Y \right) }{\Gamma \left( 1+rm+\beta_Y \right) (rm)^{\beta_X}} =1,
\end{equation*}
and the integral over region $D''$ can be estimated analogously by interchanging $s_X$ and $s_Y$ in the preceding computation, we obtain the desired limit, which is:
\begin{align*}
\lim_{m\to\infty} \frac{E^1_{rm}}{L_{rm}} =& \lim_{m\to\infty} \begin{aligned}[t] & \Gamma \left( 1+rm+\beta_X+\beta_Y \right) \\[-1ex]
 & \hspace{1em} \cdot \smashoperator{\int\limits_{D' \cup D''}} I_{rm}(s_X+s_Y) \: I_{\beta_X-1}(S_X-s_X) \: I_{\beta_Y-1}(S_Y-s_Y) \; ds_X ds_Y \end{aligned} \\
=& 0.
\end{align*}
This proves $\eqref{eq:eqE1}$.


\subsubsection*{(Checking of $\eqref{eq:eqE2}$:)}

Next, we verify that $\lim_{m\to\infty} E^2_{rm}/L_{rm} =0$. Recall that 
\begin{align*}
\frac{E^2_{rm}}{L_{rm}} =& \Gamma \left( 1+rm+\beta_X+\beta_Y \right) \\[-1ex]
 & \hspace{2em} \cdot \smashoperator[l]{\int\limits_{{[0,S_X]} \times {[0,S_Y]}}} \sum_{\substack{0 \leq k < r_X \\ 0 \leq \ell < r_Y \\ \xi_{r_X}^k \neq \xi_{r_Y}^\ell }} \begin{aligned}[t] & I_{rm}(\xi_{r_X}^k s_X+ \xi_{r_Y}^\ell s_Y) \\ & \hspace{1em} \cdot I_{\beta_X-1}(S_X-s_X) \; I_{\beta_Y-1}(S_Y-s_Y) ds_X ds_Y. \end{aligned}
\end{align*}
We then estimate the integral of one of the summands: for $0\leq k < r_X$ and $0\leq \ell < r_Y$ such that $\xi_{r_X}^k \neq \xi_{r_Y}^\ell$, we compute
\begingroup
\allowdisplaybreaks[1]
\begin{align*}
 & \Biggl\lVert \begin{aligned}[t] & \Gamma \left( 1+rm \right) (rm)^{ \beta_X+\beta_Y } \\
 & \hspace{1em} \cdot \smashoperator{\int\limits_{{[0,S_X]} \times {[0,S_Y]}}} I_{rm}(\xi_{r_X}^k s_X+ \xi_{r_Y}^\ell s_Y) \: I_{\beta_X-1}(S_X-s_X) \: I_{\beta_Y-1}(S_Y-s_Y) \; ds_X ds_Y \Biggr\rVert \end{aligned} \\[1ex]
=&(rm)^{2} \Biggl\lVert \; \int\limits_{{[0,S_X]} \times {[0,S_Y]}} \begin{aligned}[t] & \left( \xi_{r_X}^k s_X+ \xi_{r_Y}^\ell s_Y \right)^{rm} \\[-1ex]
 & \hspace{1em} \cdot I_{\beta_X-1} \left(rm(S_X-s_X) \right) \: I_{\beta_Y-1} \left( rm(S_Y-s_Y) \right) \; ds_X ds_Y \Biggr\rVert \end{aligned} \\[1ex]
=& \Biggl\lVert \; \int\limits_{{[0,rm S_X]} \times {[0,rm S_Y]}} \begin{aligned}[t] & \left( \xi_{r_X}^k \left( S_X-\frac{s'_X}{rm} \right) + \xi_{r_Y}^{\ell} \left( S_Y-\frac{s'_Y}{rm} \right) \right)^{rm} \\[-2ex]
 & \hspace{10em} \cdot I_{\beta_X-1}(s'_X) \: I_{\beta_Y-1}(s'_Y) \; ds'_X ds'_Y \Biggr\rVert, \end{aligned}
\end{align*}
\endgroup
where $s'_X=rm(S_X-s_X)$ and $s'_Y=rm(S_Y-s_Y)$.
Since $(s'_X, s'_Y)\in {[0,rmS_X]} \times {[0,rmS_Y]}$, we have
\begin{equation*}
\xi_{r_X}^k \left( S_X-\frac{s'_X}{rm} \right) + \xi_{r_Y}^{\ell} \left( S_Y-\frac{s'_Y}{rm} \right) \in \Poly \left( 0, \xi_{r_X}^k S_X, \xi_{r_Y}^\ell S_Y, \xi_{r_X}^k S_X+\xi_{r_Y}^\ell S_Y \right).
\end{equation*}
\begin{figure}[h]
\centering
\begin{tikzpicture}[scale=5]
\draw (-0.7,0)--(0.7,0);
\draw (0,-0.3)--(0,0.7);

\draw[thick] (0,0)--({0.6*cos(36)},{0.6*sin(36)});
\draw[thick] ({0.6*cos(36)},{0.6*sin(36)})--({0.6*cos(36)+0.4*cos(150)},{0.6*sin(36)+0.4*sin(150)});
\draw[thick] ({0.6*cos(36)+0.4*cos(150)},{0.6*sin(36)+0.4*sin(150)})--({0.4*cos(150)},{0.4*sin(150)});
\draw[thick] ({0.4*cos(150)},{0.4*sin(150)})--(0,0);

\node[rectangle, fill=white] at (0-0.07, 0-0.07) {$0$};
\node[rectangle, fill=white] at ({0.6*cos(36)+0.15}, {0.6*sin(36)}) {$\xi_{r_X}^k S_X$};
\node[rectangle, fill=white] at ({0.6*cos(36)+0.4*cos(150)}, {0.6*sin(36)+0.4*sin(150)+0.1}) {$\xi_{r_X}^k S_X+\xi_{r_Y}^\ell S_Y$};
\node[rectangle, fill=white] at ({0.4*cos(150)-0.15}, {0.4*sin(150)}) {$\xi_{r_Y}^\ell S_Y$};
\node[rectangle, fill=white] at (0.7, 0.7) {$\mathbb{C}$};

\end{tikzpicture}
\caption*{Figure of $\Poly \left( 0, \xi_{r_X}^k S_X, \xi_{r_Y}^\ell S_Y, \xi_{r_X}^k S_X+\xi_{r_Y}^\ell S_Y \right) \subseteq \mathbb{C}$ }
\end{figure}
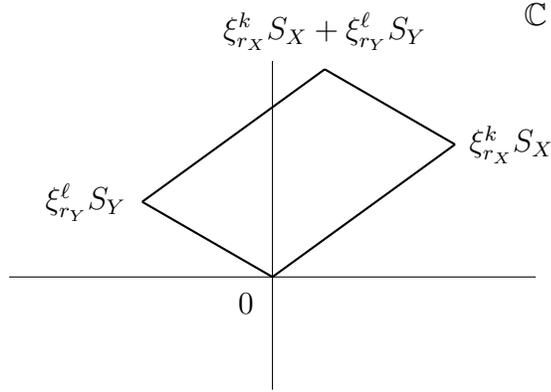
By convexity and given that $\xi_{r_X}^k \neq \xi_{r_Y}^\ell$, for all $(s_X, s_Y)\in {[0,rmS_X]} \times {[0,rmS_Y]}$, we have
\begin{equation*}
\abs*{\xi_{r_X}^k \left( S_X-\frac{s_X}{rm} \right) + \xi_{r_Y}^{\ell} \left( S_Y-\frac{s_Y}{rm} \right)} \leq \max \left( S_X, S_Y, \abs*{ \xi_{r_X}^k S_X+\xi_{r_Y}^\ell S_Y } \right) < 1.
\end{equation*}
Thus, we obtain
\begingroup
\allowdisplaybreaks[1]
\begin{align*}
 & \Biggl\lVert \begin{aligned}[t] & \Gamma \left( 1+rm \right) (rm)^{ \beta_X+\beta_Y } \\[-1ex]
 & \hspace{1em} \cdot \smashoperator{\int\limits_{{[0,S_X]} \times {[0,S_Y]}}} I_{rm}(\xi_{r_X}^k s_X+ \xi_{r_Y}^\ell s_Y) \: I_{\beta_X-1}(S_X-s_X) \: I_{\beta_Y-1}(S_Y-s_Y) \; ds_X ds_Y \Biggr\rVert \end{aligned} \\[1ex]
=& \Biggl\lVert \; \int\limits_{{[0,rm S_X]} \times {[0,rm S_Y]}} \begin{aligned}[t] & \left( \xi_{r_X}^k \left( S_X-\frac{s'_X}{rm} \right) + \xi_{r_Y}^{\ell} \left( S_Y-\frac{s'_Y}{rm} \right) \right)^{rm} \\[-2ex]
 & \hspace{10em} \cdot I_{\beta_X-1}(s'_X) \: I_{\beta_Y-1}(s'_Y) \; ds'_X ds'_Y \Biggr\rVert \end{aligned} \\[-3ex]
\leq& \overbrace{\left( \max \left( S_X, S_Y, \abs*{ \xi_{r_X}^k S_X+\xi_{r_Y}^\ell S_Y } \right) \right)^{rm}}^{\substack{\text{exponential}\\ \text{as }m\to\infty}} \hspace{-1em} \underbrace{\smashoperator[r]{\int\limits_{{[0,rm S_X]} \times {[0,rm S_Y]}}} \norm*{ I_{\beta_X-1}(s_X) \: I_{\beta_Y-1}(s_Y) } \; ds_X ds_Y}_{\substack{\text{polynomial}\\ \text{as }m\to\infty}}.
\end{align*}
\endgroup
As $\max \left( S_X, S_Y, \abs*{ \xi_{r_X}^k S_X+\xi_{r_Y}^\ell S_Y } \right)<1$, we have, for $0\leq k < r_X$ and $0\leq \ell < r_Y$ such that $\xi_{r_X}^k \neq \xi_{r_Y}^\ell$,
\begin{equation*}
\lim_{m\to\infty} \begin{aligned}[t] & \Gamma \left( 1+rm \right) (rm)^{ \beta_X+\beta_Y } \\
 & \hspace{1em} \cdot \smashoperator{\int\limits_{{[0,S_X]} \times {[0,S_Y]}}} I_{rm}(\xi_{r_X}^k s_X+ \xi_{r_Y}^\ell s_Y) \: I_{\beta_X-1}(S_X-s_X) \: I_{\beta_Y-1}(S_Y-s_Y) \; ds_X ds_Y =0 \end{aligned}
\end{equation*}
Given that
\begin{equation*}
\lim_{m\to\infty} \frac{\Gamma \left( 1+rm+\beta_X+\beta_Y \right) }{\Gamma \left( 1+rm \right) (rm)^{ \beta_X+\beta_Y }} =1,
\end{equation*}
and summing the above estimate over $k$ and $\ell$, we obtain
\begin{align*}
\lim_{m\to\infty} \frac{E^2_{rm}}{L_{rm}} =& \lim_{m\to\infty} \Gamma \left( 1+rm+\beta_X+\beta_Y \right) \\[-1ex]
 & \hspace{3em} \cdot \smashoperator[l]{\int\limits_{{[0,S_X]} \times {[0,S_Y]}}} \sum_{\substack{0 \leq k < r_X \\ 0 \leq \ell < r_Y \\ \xi_{r_X}^k \neq \xi_{r_Y}^\ell }} \begin{aligned}[t] & I_{rm}(\xi_{r_X}^k s_X+ \xi_{r_Y}^\ell s_Y) \\
 & \hspace{1em} \cdot I_{\beta_X-1}(S_X-s_X) \; I_{\beta_Y-1}(S_Y-s_Y) ds_X ds_Y \end{aligned} \\
=& 0.
\end{align*}
This proves $\eqref{eq:eqE2}$.


\subsubsection*{(Checking of $\eqref{eq:eqE3}$:)}

We now proceed to calculate the term $E^3_{rm}$. Recall that
\begin{equation*}
E^3_{rm} \defeq \hspace{-2em} \sum_{\substack{(m_X, m_Y) \in \mathbb{N}^2 \\ r_X m_X + r_Y m_Y = rm}} \hspace{-1em} \begin{aligned}[t] & \left( A_Y R^X_{r_X m_X} + A_X R^Y_{r_Y m_Y} + R^X_{r_X m_X} R^Y_{r_Y m_Y} \right) \\ & \hspace{1em} \cdot \smashoperator{\int\limits_{{[0,S_X]} \times {[0,S_Y]}}} I_{r_X m_X}(s_X) I_{r_Y m_Y}(s_Y) I_{\beta_X-1}(S_X-s_X) I_{\beta_Y-1}(S_Y-s_Y) \: ds_X ds_Y. \end{aligned}
\end{equation*}
Define $\mathcal{E}_{rm}(s_X, s_Y)$ and $\mathcal{E}(r_X m_X, r_Y m_Y)$ as follows:
\begin{equation*}
E^3_{rm} = \hspace{-1em} \int\limits_{{[0,S_X]} \times {[0,S_Y]}} \hspace{-1em} \mathcal{E}_{rm}(s_X, s_Y) \; ds_X ds_Y,
\end{equation*}
\begin{equation*}
\mathcal{E}_{rm}(s_X, s_Y) \defeq \hspace{-1em} \sum_{\substack{(m_X, m_Y) \in \mathbb{N}^2 \\ r_X m_X + r_Y m_Y = rm}} \hspace{-1em} \mathcal{E}(r_X m_X, r_Y m_Y),
\end{equation*}
\vspace{-2ex}
\begin{align*}
\mathcal{E}(r_X m_X, r_Y m_Y) \defeq & \left( A_Y R^X_{r_X m_X} + A_X R^Y_{r_Y m_Y} + R^X_{r_X m_X} R^Y_{r_Y m_Y} \right) \\
& \hspace{1em} \cdot I_{r_X m_X}(s_X) \: I_{r_Y m_Y}(s_Y) \: I_{\beta_X-1}(S_X-s_X) \: I_{\beta_Y-1}(S_Y-s_Y).
\end{align*}
Note that $\mathcal{E}(r_X m_X, r_Y m_Y)$ depends on $(s_X, s_Y)$. For simplicity, We suppress this dependence in the notation for simplicity.
Let $\mathring{\mathcal{E}}_{rm, M_X, M_Y}(s_X, s_Y)$ denotes the sum of $\mathcal{E}(r_X m_X, r_Y m_Y)$ over the interior of line $r_X m_X + r_Y m_Y = rm$, specifically:
\begin{equation*}
\mathring{\mathcal{E}}_{rm, M_X, M_Y}(s_X, s_Y) \defeq \hspace{-1em} \sum_{\substack{(m_X, m_Y) \in \mathbb{N}^2 \\ r_X m_X + r_Y m_Y = rm \\ m_X \geq M_X, m_Y \geq M_Y}} \hspace{-1em} \mathcal{E}(r_X m_X, r_Y m_Y).
\end{equation*}
Subsequently, we decompose $E^3_{rm}$ into the two terms, $E^4_{rm, M_X, M_Y}$ and $E^5_{rm, M_X, M_Y}$, corresponding to the decomposition of $\mathcal{E}$ into boundary part and interior parts, respectively:
\begin{align*}
 E^3_{rm} =& \hspace{-1em} \int\limits_{{[0,S_X]} \times {[0,S_Y]}} \hspace{-1em} \mathcal{E}_{rm}(s_X, s_Y) \; ds_X ds_Y \\
=& \hspace{-1em} \left. \int\limits_{{[0,S_X]} \times {[0,S_Y]}} \hspace{-1em} \left( \mathcal{E}_{rm}(s_X, s_Y) - \mathring{\mathcal{E}}_{rm, M_X, M_Y}(s_X, s_Y) \right) ds_X ds_Y \right\} \defeq E^4_{rm, M_X, M_Y} \\
 & \hspace{1em} + \hspace{-1em} \left. \int\limits_{{[0,S_X]} \times {[0,S_Y]}} \hspace{-1em} \mathring{\mathcal{E}}_{rm, M_X, M_Y}(s_X, s_Y) \; ds_X ds_Y. \right\} \defeq E^5_{rm, M_X, M_Y}
\end{align*}
We claim that $E^4_{rm,M_X,M_Y}$ and $E^5_{rm,M_X,M_Y}$ exhibit the following asymptotic properties relative to $L_{rm}$:
\begin{itemize}
\item for all $M_X, M_Y \in \mathbb{N}$, we have 
\begin{equation*}
\lim_{m\to\infty} \frac{E^4_{rm,M_X,M_Y}}{L_{rm}} =0, \tag{$E4$} \label{eq:eqE4}
\end{equation*}
\item for all $\varepsilon >0$, there exist $M_X, M_Y\in\mathbb{N}$ such that
\begin{equation*}
\limsup_{m\to\infty} \norm*{\frac{E^5_{rm,M_X,M_Y}}{L_{rm}}} < \varepsilon. \tag{$E5$} \label{eq:eqE5}
\end{equation*}
\end{itemize}
Assuming these two claims hold, we can then combine them to estimate $E^3_{rm}$. Given an arbitrary $\varepsilon >0$, we select $M_X, M_Y$ as specified above, resulting in:
\begin{align*}
& \limsup_{m\to\infty} \norm*{\frac{E^3_{rm}}{L_{rm}}} \\
=& \limsup_{m\to\infty} \norm*{\frac{E^4_{rm, M_X, M_Y} + E^5_{rm, M_X, M_Y}}{L_{rm}}} \\
\leq& \limsup_{m\to\infty} \norm*{\frac{E^5_{rm, M_X, M_Y}}{L_{rm}}} < \varepsilon.
\end{align*}
Since $\varepsilon >0$ is arbitrary, we conclude that the desired asymptotic property of $E^3_{rm}$ holds, given $\eqref{eq:eqE4}$ and $\eqref{eq:eqE5}$:
\begin{equation*}
\lim_{m\to\infty} \frac{E^3_{rm}}{L_{rm}} = \limsup_{m\to\infty} \norm*{\frac{E^3_{rm}}{L_{rm}}} = 0.
\end{equation*}


\subsubsection*{(Checking of $\eqref{eq:eqE4}$:)}
We now proceed to prove $\eqref{eq:eqE4}$.
For $m_X^\circ \in \mathbb{N}$, let $\mathcal{E}'_{rm}(m_X^\circ)$ denote the value of $\mathcal{E}(m_X^\circ, m_Y)$ for the solution to the Diophantine equation $r_X m_X^\circ + r_Y m_Y = rm$, and $0$ no solution exists. Simlarly, for $m_Y^\circ \in \mathbb{N}$, let $\mathcal{E}''_{rm}(m_Y^\circ)$ denote the value of $\mathcal{E}(m_X, m_Y^\circ)$ for the solution to the Diophantine equation $r_X m_X + r_Y m_Y^\circ = rm$. In other words, $\mathcal{E}'_{rm}(m_X^\circ)$ and $\mathcal{E}''_{rm}(m_Y^\circ)$ are defined by:
\begin{equation*}
\mathcal{E}'_{rm}(m_X^\circ) = \begin{cases*} 
\mathcal{E}(r_X m_X^\circ, rm-r_X m_X^\circ) & if $r_Y \mid (rm-r_X m_X^\circ)$, \\
0 & otherwise,
\end{cases*}
\end{equation*}
\begin{equation*}
\mathcal{E}''_{rm}(m_Y^\circ) = \begin{cases*}
\mathcal{E}(rm-r_Y m_Y^\circ, r_Y m_Y^\circ) & if $r_X \mid (rm-r_Y m_Y^\circ)$, \\
0 & otherwise.
\end{cases*}
\end{equation*}
Due to the identity
\begin{equation*}
\sum_{\substack{(m_X, m_Y) \in \mathbb{N}^2 \\ r_X m_X + r_Y m_Y = rm \\ m_X < M_X}} \hspace{-1em} \mathcal{E}(r_X m_X, r_Y m_Y)=  \sum_{m_X^\circ=0}^{M_X-1} \mathcal{E}'_{rm}(m_X^\circ),
\end{equation*}
we can express $\mathcal{E}_{rm}(s_X, s_Y)-\mathring{\mathcal{E}}_{rm, M_X, M_Y}(s_X, s_Y)$ as a finite sum. Specifically, for large $m$ such that all three index sets
\begin{itemize} 
\item $\{(m_X, m_Y)\in\mathbb{N} \mid r_X m_X + r_Y m_Y = rm, m_X < M_X\}$, 
\item $\{(m_X, m_Y)\in\mathbb{N} \mid r_X m_X + r_Y m_Y = rm, m_X \geq M_X, m_Y \geq M_Y\}$, and 
\item $\{(m_X, m_Y)\in\mathbb{N} \mid r_X m_X + r_Y m_Y = rm, m_Y < M_Y\}$
\end{itemize}
are pairwise disjoint, we have:
\begin{equation*}
\mathcal{E}_{rm}(s_X, s_Y) = \sum_{m_X^\circ=0}^{M_X-1} \mathcal{E}'_{rm}(m_X^\circ) + \mathring{\mathcal{E}}_{rm, M_X, M_Y}(s_X, s_Y) + \sum_{m_Y^\circ=0}^{M_Y-1} \mathcal{E}''_{rm}(m_Y^\circ).
\end{equation*}
Therefore,
\begin{align*}
E^4_{rm} =& \hspace{-1em} \int\limits_{{[0,S_X]} \times {[0,S_Y]}} \hspace{-1em} \left( \mathcal{E}_{rm}(s_X, s_Y) - \mathring{\mathcal{E}}_{rm, M_X, M_Y}(s_X, s_Y) \right) ds_X ds_Y \\
=& \hspace{-1em} \int\limits_{{[0,S_X]} \times {[0,S_Y]}} \hspace{-0.5em} \left( \sum_{m_X^\circ=0}^{M_X-1} \mathcal{E}'_{rm}(m_X^\circ) + \sum_{m_Y^\circ=0}^{M_Y-1} \mathcal{E}''_{rm}(m_Y^\circ) \right) ds_X ds_Y.
\end{align*}
Since $E^4_{rm}$ can be expressed as a finite sum of integrals, we compare the asymptotic behaviour of each summand relative to $L_{rm}$.
For each $\mathcal{E}'_{rm}(m_X^\circ)$, consider the non-zero subsequence, defined as $\mathcal{E}(r_X m_X^\circ, rm-r_X m_X^\circ)$:
\begingroup
\allowdisplaybreaks[1]
\begin{align*}
 & \frac{1}{L_{rm}} \int\limits_{{[0,S_X]} \times {[0,S_Y]}} \hspace{-1em} \mathcal{E} \left( r_X m_X^\circ, rm - r_X m_X^\circ \right) \; ds_X ds_Y \\[1ex]
=& \Gamma \left( 1+ rm + \beta_X + \beta_Y \right) \: I_{r_X m_X^\circ+\beta_X}(S_X) \: I_{rm - r_X m_X^\circ+\beta_Y}(S_Y) \\
 & \hspace{1em} \cdot \left( A_Y R^X_{r_X m_X^\circ} + A_X R^Y_{rm - r_X m_X^\circ} + R^X_{r_X m_X^\circ} R^Y_{rm - r_X m_X^\circ} \right) \\[2ex]
=& \overbrace{S_Y^{rm - r_X m_X^\circ+\beta_Y}}^{\text{exponential as }m\to\infty} \cdot \overbrace{\frac{\Gamma \left( 1 + rm + \beta_X + \beta_Y \right)}{\Gamma \left( 1 + rm - r_X m_X^\circ+\beta_Y \right)}}^{\text{polynomial as }m\to\infty} \\[1ex]
 & \hspace{1em} \cdot \underbrace{\left( A_Y R^X_{r_X m_X^\circ} + A_X R^Y_{rm - r_X m_X^\circ} + R^X_{r_X m_X^\circ} R^Y_{rm - r_X m_X^\circ} \right) I_{r_X m_X^\circ + \beta_X}(S_X)}_{\text{converge as }m\to\infty}.
\end{align*}
\endgroup
As $S_Y < 1$, we have, for all $m_X^\circ \in \mathbb{N}$,
\begin{equation*}
\lim_{m\to\infty} \frac{1}{L_{rm}} \int\limits_{{[0,S_X]} \times {[0,S_Y]}} \hspace{-1em} \mathcal{E}'_{rm}(m_X^\circ) \; ds_X ds_Y =0.
\end{equation*}
Similarly, for all $m_Y^\circ \in \mathbb{N}$,
\begin{equation*}
\lim_{m\to\infty} \frac{1}{L_{rm}} \int\limits_{{[0,S_X]} \times {[0,S_Y]}} \hspace{-1em} \mathcal{E}''_{rm}(m_Y^\circ) \; ds_X ds_Y =0.
\end{equation*}
By summing the above limits, we have, for arbitrary $M_X, M_Y \in \mathbb{N}$, 
\begin{align*}
\lim_{m\to\infty} \frac{E^4_{rm, M_X, M_Y}}{L_{rm}} =& \frac{1}{L_{rm}} \int\limits_{{[0,S_X]} \times {[0,S_Y]}} \hspace{-0.5em} \left( \sum_{m_X^\circ=0}^{M_X-1} \mathcal{E}'_{rm}(m_X^\circ) + \sum_{m_Y^\circ=0}^{M_Y-1} \mathcal{E}''_{rm}(m_Y^\circ) \right) ds_X ds_Y \\
=& 0.
\end{align*}
This proves $\eqref{eq:eqE4}$.


\subsubsection*{(Checking of $\eqref{eq:eqE5}$:)}

Finally, $\eqref{eq:eqE5}$ is the last statement we need to check. Recall that 
\begin{align*}
\frac{E^5_{rm,M_X,M_Y}}{L_{rm}} =& \Gamma \left( 1+rm+\beta_X+\beta_Y \right) \\
& \hspace{1em} \cdot \hspace{-1em} \int\limits_{{[0,S_X]} \times {[0,S_Y]}} \sum_{\substack{(m_X, m_Y) \in \mathbb{N}^2 \\ r_X m_X + r_Y m_Y = rm \\ m_X \geq M_X, m_Y \geq M_Y}} \hspace{-1em} \mathcal{E}(r_X m_X, r_Y m_Y) \; ds_X ds_Y .
\end{align*}
We compute
\begingroup
\allowdisplaybreaks[1]
\begin{align*}
& \Bigg\lVert \Gamma \left( 1+rm \right) (rm)^{\beta_X+\beta_Y} \hspace{-1em} \int\limits_{{[0,S_X]} \times {[0,S_Y]}} \sum_{\substack{(m_X, m_Y) \in \mathbb{N}^2 \\ r_X m_X + r_Y m_Y = rm \\ m_X \geq M_X, m_Y \geq M_Y}} \hspace{-1em} \mathcal{E}(r_X m_X, r_Y m_Y) \; ds_X ds_Y \Bigg\rVert \\[2ex]
\leq& (rm)^2 \hspace{-1em} \max_{\substack{(m_X, m_Y) \in \mathbb{N}^2 \\ r_X m_X + r_Y m_Y = rm \\ m_X \geq M_X, m_Y \geq M_Y}} \norm*{A_Y R^X_{r_X m_X} + A_X R^Y_{r_Y m_Y} + R^X_{r_X m_X} R^Y_{r_Y m_Y}} \\
 & \hspace{1em} \cdot \hspace{-1em} \int\limits_{{[0,S_X]} \times {[0,S_Y]}} \hspace{-1em} \begin{aligned}[t] & \Gamma \left( 1+rm \right) \hspace{-1em} \sum_{\substack{(m_X, m_Y) \in \mathbb{N}^2 \\ r_X m_X + r_Y m_Y = rm}} \hspace{-1em} I_{r_X m_X}(s_X) \: I_{r_Y m_Y}(s_Y) \\
 & \hspace{1em} \cdot \norm*{ I_{\beta_X-1} \left( rm(S_X-s_X) \right) \: I_{\beta_Y-1} \left( rm(S_Y-s_Y) \right) } \; ds_X ds_Y \end{aligned} \\[2ex]
\leq& \hspace{-1em} \max_{\substack{(m_X, m_Y) \in \mathbb{N}^2 \\ r_X m_X + r_Y m_Y = rm \\ m_X \geq M_X, m_Y \geq M_Y}} \norm*{A_Y R^X_{r_X m_X} + A_X R^Y_{r_Y m_Y} + R^X_{r_X m_X} R^Y_{r_Y m_Y}} \\
 & \hspace{1em} \cdot \frac{1}{r_X r_Y} \hspace{-1em} \int\limits_{{[0,rm S_X]} \times {[0,rm S_Y]}} \sum_{\substack{0 \leq k < r_X \\ 0 \leq \ell < r_Y}} \begin{aligned}[t] & \left( \xi_{r_X}^k \left(S_X-\frac{s'_X}{rm} \right)+ \xi_{r_Y}^\ell \left(S_Y-\frac{s'_Y}{rm} \right) \right)^{rm} \\
 & \hspace{5em} \cdot \norm*{ I_{\beta_X-1}(s'_X) \: I_{\beta_Y-1}(s'_Y) } \; ds'_X ds'_Y, \end{aligned}
\end{align*}
\endgroup
where $s'_X=rm(S_X-s_X)$ and $s'_Y=rm(S_Y-s_Y)$.
Using the same method employed during the computation for $\eqref{eq:eqE2}$, we can show that 
\begin{equation*}
\lim_{m\to\infty} \hspace{-1em} \int\limits_{{[0,rm S_X]} \times {[0,rm S_Y]}} \sum_{\substack{0 \leq k < r_X \\ 0 \leq \ell < r_Y \\ \xi_{r_X}^k \neq \xi_{r_Y}^\ell }} \begin{aligned}[t] & \left( \xi_{r_X}^k \left(S_X-\frac{s_X}{rm} \right)+ \xi_{r_Y}^\ell \left(S_Y-\frac{s_Y}{rm} \right) \right)^{rm} \\
 & \hspace{5em} \cdot \norm*{ I_{\beta_X-1}(s_X) \: I_{\beta_Y-1}(s_Y) } \; ds_X ds_Y = 0. \end{aligned}
\end{equation*}
The remaining term can be estimated as follows:
\begingroup
\allowdisplaybreaks[1]
\begin{align*}
 & \hspace{-1em} \max_{\substack{(m_X, m_Y) \in \mathbb{N}^2 \\ r_X m_X + r_Y m_Y = rm \\ m_X \geq M_X, m_Y \geq M_Y}} \norm*{A_Y R^X_{r_X m_X} + A_X R^Y_{r_Y m_Y} + R^X_{r_X m_X} R^Y_{r_Y m_Y}} \\
 & \hspace{1em} \cdot \frac{r}{r_X r_Y} \hspace{-1em} \int\limits_{{[0,rm S_X]} \times {[0,rm S_Y]}} \hspace{-0.5em} \left( S_X-\frac{s'_X}{rm} + S_Y-\frac{s'_Y}{rm} \right)^{rm} \hspace{-0.5em} \norm*{ I_{\beta_X-1}(s'_X) \: I_{\beta_Y-1}(s'_Y) } \; ds'_X ds'_Y \\[3ex]
\leq& \hspace{-1em} \max_{\substack{(m_X, m_Y) \in \mathbb{N}^2 \\ r_X m_X + r_Y m_Y = rm \\ m_X \geq M_X, m_Y \geq M_Y}} \norm*{A_Y R^X_{r_X m_X} + A_X R^Y_{r_Y m_Y} + R^X_{r_X m_X} R^Y_{r_Y m_Y}} \\
 & \hspace{1em} \cdot \frac{r}{r_X r_Y} \hspace{-1em} \int\limits_{{[0,rm S_X]} \times {[0,rm S_Y]}} \hspace{-1em} e^{-(s'_X+s'_Y)} \norm*{ I_{\beta_X-1}(s'_X) \: I_{\beta_Y-1}(s'_Y) } \; ds'_X ds'_Y.
\end{align*}
\endgroup
Since the integral in the final line above converges, let $C$ denote its value, i.e.:
\begin{equation*}
\int\limits_{{[0,\infty)} \times {[0,\infty)}} \hspace{-1em} e^{-(s'_X+s'_Y)} \norm*{ I_{\beta_X-1}(s'_X) \: I_{\beta_Y-1}(s'_Y) } \; ds'_X ds'_Y = C < \infty.
\end{equation*}
Let $\varepsilon >0$ be arbitrary. Since $\lim_{m_X \to\infty} R^X_{r_X m_X}=\lim_{m_Y \to\infty} R^Y_{r_Y m_Y}=0$, there exist $M_X, M_Y \in\mathbb{N}$ such that for all $m_X \geq M_X$ and $m_Y \geq M_Y$, 
\begin{equation*}
\norm*{A_Y R^X_{r_X m_X} + A_X R^Y_{r_Y m_Y} + R^X_{r_X m_X} R^Y_{r_Y m_Y}} < \frac{r_X r_Y}{r} \: \frac{1}{C} \: \varepsilon.
\end{equation*}
Using
\begin{equation*}
\lim_{m\to\infty} \frac{\Gamma \left( 1+rm+\beta_X+\beta_Y \right)}{\Gamma \left( 1+rm \right) (rm)^{ \beta_X+\beta_Y }} =1,
\end{equation*}
we have verified our final estimation $\eqref{eq:eqE5}$, which state that for all $\varepsilon >0$, there exist $M_X, M_Y \in\mathbb{N}$ such that
\begin{equation*}
\limsup_{m\to\infty} \norm*{\frac{E^5_{rm, M_X, M_Y}}{L_{rm}}} < \varepsilon.
\end{equation*}
This proves $\eqref{eq:eqE5}$.


\subsubsection*{(Combining all the results:)}

Combining the equations $\eqref{eq:eqA}$, $\eqref{eq:eqL}$, $\eqref{eq:eqE1}$, $\eqref{eq:eqE2}$ and $\eqref{eq:eqE3}$, we obtain the following as $m\to\infty$:
\begin{align*}
(\alpha^X \alpha^Y)_{rm} =& \hspace{-1em} \sum_{\substack{(m_X, m_Y) \in \mathbb{N}^2 \\ r_X m_X + r_Y m_Y = rm}} \hspace{-1em} \alpha^X_{r_X m_X} \alpha^Y_{r_Y m_Y} \\[1ex]
=& \left( \frac{r A_X A_Y}{r_X r_Y} L_{rm} - \frac{r A_X A_Y}{r_X r_Y} E^1_{rm} + \frac{A_X A_Y}{r_X r_Y} E^2_{rm} + E^3_{rm} \right) \; (T_X+T_Y)^{rm+\beta_X+\beta_Y} \\[1ex]
=& L_{rm} \; (T_X+T_Y)^{rm+\beta_X+\beta_Y} \cdot \left( \frac{r A_X A_Y}{r_X r_Y} +\mathit{o}(1) \right) \\[1ex]
=& \frac{(T_X+T_Y)^{rm+\beta_X+\beta_Y}}{\Gamma\left(1+rm+\beta_X+\beta_Y \right)} \cdot \left( \frac{r A_X A_Y}{r_X r_Y}+\mathit{o}(1) \right).
\end{align*}
Therefore, $\alpha^X(t) \alpha^Y(t)$ is $\left( T_X+T_Y, r \frac{A_X}{r_X} \frac{A_Y}{r_Y} \right)$-scaled asymptotically Mittag-Leffler.


\subsubsection*{(Extending to general $\beta$:)}

We now extend the result to general $\beta_X$ and $\beta_Y$. Write $\alpha^X(t)$ and $\alpha^Y(t)$ in the form
\begin{equation*}
\alpha^X(t) = \sum_{m_X=0}^{M_X} \alpha^X_{r_X m_X} t^{r_X m_X+\beta_X} + \sum_{m_X=0}^\infty \alpha^X_{r_X (m_X+M_X)} t^{r_X m_X+(r_X M_X + \beta_X)},
\end{equation*}
\begin{equation*}
\alpha^Y(t) = \sum_{m_Y=0}^{M_Y} \alpha^Y_{r_Y m_Y} t^{r_Y m_Y+\beta_Y} + \sum_{m_Y=0}^\infty \alpha^Y_{r_Y (m_Y+M_Y)} t^{r_X m_Y+(r_Y M_Y + \beta_X)}.
\end{equation*}
so that the $H^0$ components of both $r_X M_X + \beta_X$ and $r_Y M_Y + \beta_Y$ have positive real parts. Then, the coefficients of the product $\alpha^X(t)\alpha^Y(t)$ can be expressed as:
\begin{align*}
 & (\alpha^X\alpha^Y)_{rm} \\
=& \Biggl( \sum_{\substack{\substack{(m_X, m_Y) \in \mathbb{N}^2 \\ r_X m_X + r_Y m_Y = rm \\ m_X < M_X}}} + \sum_{\substack{\substack{(m_X, m_Y) \in \mathbb{N}^2 \\ r_X m_X + r_Y m_Y = rm \\ m_X \geq M_X, m_Y \geq M_Y}}} + \sum_{\substack{\substack{(m_X, m_Y) \in \mathbb{N}^2 \\ r_X m_X + r_Y m_Y = rm \\ m_Y < M_Y}}} \Biggr) \alpha^X_{r_X m_X} \alpha^Y_{r_Y m_Y}.
\end{align*}
Since the middle summation is the sum of coefficients of the product of the tails of the original series, and the property of being asymptotically Mittag-Leffler is preserved under taking tails (Proposition \ref{propAML} (4)), we have, as $m\to\infty$:
\begin{equation*}
\sum_{\substack{\substack{(m_X, m_Y) \in \mathbb{N}^2 \\ r_X m_X + r_Y m_Y = rm \\ m_X \geq M_X, m_Y \geq M_Y}}} \hspace{-1em} \alpha^X_{r_X m_X} \alpha^Y_{r_Y m_Y} = \frac{(T_X+T_Y)^{rm+\beta_X+\beta_Y}}{\Gamma \left( 1+rm+\beta_X+\beta_Y \right)} \cdot \left( r\frac{A_X}{r_X}\frac{A_Y}{r_Y}+\mathit{o}(1) \right).
\end{equation*}
Estimations of the left and right sums are performed in a similar fashion to the computation of $\eqref{eq:eqE4}$. By expressing each sum into a finite sum over fixed $m_X=m_X^\circ$, and then comparing the asymptotic behaviour of each summand with the desired asymptotic, we have, as $m\to\infty$:
\begin{align*}
\sum_{\substack{\substack{(m_X, m_Y) \in \mathbb{N}^2 \\ r_X m_X + r_Y m_Y = rm \\ m_X < M_X}}} \hspace{-1em} \alpha^X_{r_X m_X} \alpha^Y_{r_Y m_Y} =& \hspace{-1em} \sum_{\substack{0 \leq m_X^\circ < M_X \\ r_Y \mid (rm-r_X m_X^\circ)}} \hspace{-1em} \alpha^X_{r_X m_X^\circ} \alpha^Y_{rm-r_X m_X^\circ} \\
=& \hspace{-1em} \sum_{\substack{0 \leq m_X^\circ < M_X \\ r_Y \mid (rm-r_X m_X^\circ)}} \hspace{-1em} \alpha^X_{r_X m_X^\circ} \frac{T_Y^{rm-r_X m_X^\circ+\beta_Y}}{\Gamma \left( 1+rm-r_X m_X^\circ+\beta_Y \right)} \cdot \left( A_Y + \mathit{o}(1) \right) \\[2ex]
=& \frac{(T_X+T_Y)^{rm+\beta_X+\beta_Y}}{\Gamma \left( 1+rm+\beta_X+\beta_Y \right)} \cdot \mathit{o}(1).
\end{align*}
Summing all three asymptotics, we obtain our proposed asymptotic behaviour for the coefficients of the product of two asymptotically Mittag-Leffler series with general cohomology class $\beta$, as $m\to\infty$:
\begin{equation*}
(\alpha^X\alpha^Y)_{rm} = \frac{(T_X+T_Y)^{rm+\beta_X+\beta_Y}}{\Gamma \left( 1+rm+\beta_X+\beta_Y \right)} \cdot \left( r\frac{A_X}{r_X}\frac{A_Y}{r_Y}+\mathit{o}(1) \right).
\end{equation*}

\end{proof}


\subsubsection*{Application to Product of Manifolds}

We can then apply the result to the product of two Fano manifolds. 

\begin{corollary}\label{corDisGammaIProduct}

Let $X$, $Y$ be two Fano manifolds. Let $J^X(t)$ and $J^Y(t)$ be the $J$-functions of $X$ and $Y$, respectively, (Definition \ref{defJFunction}). If $t^{\frac{1}{2}\dim X} J^X(t)$ and $t^{\frac{1}{2}\dim Y}J^Y(t)$ are both asymptotically Mittag-Leffler (Definition \ref{defaML}), with scalings $(T_X, A_X)$ and $(T_Y, A_Y)$, respectively, then $t^{\frac{1}{2}\dim (X\times Y)} J^{X \times Y}(t)$ is $(T_X+T_Y, r \frac{A_X}{r_X} \frac{A_Y}{r_Y})$-scaled asymptotically Mittag-Leffler, where $r_X$ and $r_Y$ are the Fano indices of $X$ and $Y$, respectively, and $r=\mathrm{GCD} (r_X, r_Y)$.

\end{corollary}

\begin{proof}
This follows directly from the proposition by setting $\alpha^X(t) = t^{\frac{1}{2}\dim X} J^X(t)$ and $\alpha^Y(t) = t^{\frac{1}{2}\dim Y} J^Y(t)$.
\end{proof}

Note that when $X$ and $Y$ both satisfy the Gamma conjecture I, i.e., $A_X \propto \hat{\Gamma}_X$ and $A_Y \propto \hat{\Gamma}_Y$, then $A_{X \times Y} \propto \hat{\Gamma}_X \hat{\Gamma}_Y = \hat{\Gamma}_{X \times Y}$. In this case, $X \times Y$ also satisfies the Gamma conjecture I. 
}


\section{Hypersurface}\label{sec:Hypersurface}
{
It is well-established that if a Fano manifold $Z$ is realized as a section of a line bundle proportional to the anti-canonical bundle of another Fano manifold $X$, its $J$-function can be determined by the quantum Lefschetz theorem \cite{MR2276766}. In this section, we aim to prove that the $J$-function, as given by the quantum Lefschetz theorem, is also asymptotically Mittag-Leffler.

\begin{proposition}\label{DisGammaIHypersurface}

Let $X$ be a Fano manifold with $J^X(t)=J^X(c_1(X) \log t, 1)$ denoting its $J$-function, and let $r_X$ represent the Fano index of $X$ (Definition \ref{defJFunction}). Define $x=\frac{1}{r_X} c_1(X) \in H^2(X,\mathbb{Z})$. Suppose that $t^{\frac{1}{2}\dim X}J^X(t)$ is $(T_X,A_X)$-scaled asymptotically Mittag-Leffler (Definition \ref{defaML}). If $Z$ is a degree $d$ Fano hypersurface within the linear system $\abs{d x}$, where $(0<d<r_X)$, then $t^{\frac{1}{2}\dim Z}J^Z(t)$ is also asymptotically Mittag-Leffler, with scaling $(T_Z,A_Z)$ given by
\begin{equation*}
T_Z = \left( T_X^{r_X} \; \frac{r_Z^{r_Z} \; d^d}{r_X^{r_X}} \right)^{\frac{1}{r_Z}}-c_0, \quad \left( \text{equivalently, } \left( \frac{T_Z+c_0}{r_Z} \right)^{r_Z} = d^d \; \left( \frac{T_X}{r_X} \right)^{r_X} \right)
\end{equation*}
and
\begin{equation*}
A_Z \propto \frac{A_X}{\Gamma \left( 1+dx \right)},
\end{equation*}
where 
\begin{equation*}
c_0 = \begin{dcases*} d! \; \langle [\mathsf{pt}], J^{X}_{r_X} \rangle & when $r_Z=r_X-d=1$, \\ 0 & when $r_Z=r_X-d>1$. \end{dcases*}
\end{equation*}

\end{proposition}

\begin{proof}

The $J$-function of $Z$ is given by the quantum Lefschetz theorem \cite{MR2276766, MR1839288}:
\begin{equation*}
J(t) = e^{-c_0 t} \; t^{(r_X-d)x} \; \left( \sum_{m=0}^{\infty} \: \prod_{k=1}^{dm} (dx+k) \: J^{X}_{r_X m} \: t^{(r_X-d)m} \right),
\end{equation*}
where 
\begin{equation*}
c_0 = \begin{dcases*} d! \; \langle [\mathsf{pt}], J^{X}_{r_X} \rangle & when $r_Z=r_X-d=1$, \\ 0 & when $r_Z=r_X-d>1$. \end{dcases*}
\end{equation*}


\subsubsection*{(Case 1: $r_Z=r_X-d>1$)}

In the case where $r_Z=r_X-d>1$, we have
\begin{equation*}
J^Z_{r_Z m} = \frac{\Gamma(1+dm+dx)}{\Gamma(1+dx)} \; J^{X}_{r_X m}.
\end{equation*}
Given that $t^{\frac{1}{2} \dim X}J^X(t)$ is asymptotically Mittag-Leffler, the coefficients $J^X_{r_Y m}$ have the asymptotic equivlance as $m\to\infty$:
\begin{equation*}
J^{X}_{r_X m} = \frac{T_X^{r_X m + \frac{1}{2} \dim X + c_1(X)}}{\Gamma \left( 1 + r_X m + \frac{1}{2} \dim X + c_1(X) \right)} \cdot \left( A_X + \mathit{o}(1) \right).
\end{equation*}
Consequently, we obtain
\begingroup
\allowdisplaybreaks[1]
\begin{align*}
J^Z_{r_Z m} =& \; \frac{\Gamma(1+dm+dx)}{\Gamma(1+dx)} \; J^{X}_{r_X m} \\
=& \; \frac{1}{\Gamma \left( 1 + r_Z m + \frac{1}{2} \dim Z + c_1(Z) \right)} \; \frac{1}{\Gamma(1+dx)} \\
& \; \cdot T_X^{r_X m + \frac{1}{2} \dim X + c_1(X)} \; \frac{\Gamma \left( 1 + r_Z m + \frac{1}{2} \; \dim Z + c_1(Z) \right) \Gamma(1+dm+dx)}{\Gamma \left( 1 + r_X m + \frac{1}{2} \dim X + c_1(X) \right)} \\
& \hspace{24em} \cdot \left( A_X + \mathit{o}(1) \right) \\[3ex]
=& \; \frac{1}{\Gamma \left( 1 + r_Z m + \frac{1}{2} \dim Z + c_1(Z) \right)} \; \frac{1}{\Gamma(1+dx)} \; T_X^{r_X m + \frac{1}{2} \dim X + c_1(X)}\\
& \; \cdot \frac{(2\pi)^{\frac{1}{2}} \; \left(r_Z m \right)^{\frac{1}{2} + r_Z m + \frac{1}{2} \dim Z + r_Z x} \; e^{-r_Z m} \; (2\pi)^{\frac{1}{2}} \; \left( dm \right)^{\frac{1}{2} + dm+dx} \; e^{-dm}}{(2\pi)^{\frac{1}{2}} \; \left((r_Z+d) m\right)^{\frac{1}{2} + (r_Z+d) m + \frac{1}{2} (\dim Z+1) + (r_Z+d)x} \; e^{-(r_Z+d) m}} \\
& \hspace{24em} \cdot \left( A_X + \mathit{o}(1) \right) \\[3ex]
=& \frac{1}{\Gamma \left( 1 + r_Z m + \frac{1}{2} \dim Z + c_1(Z) \right)} \; \left( \left( T_X^{r_Z+d} \; \frac{r_Z^{r_Z} d^d}{(r_Z+d)^{r_Z+d}} \right)^{\frac{1}{r_Z}} \right)^{r_Z m + \frac{1}{2} \dim Z + c_1(Z)} \\
& \; \cdot (2\pi)^\frac{1}{2} \; \frac{r_Z^{\frac{1}{2}+\frac{1}{2}\dim Z} \; d^\frac{1}{2}}{(r_Z+d)^{\frac{1}{2}+\frac{1}{2}\dim Z+\frac{1}{2}}} \; T_X^{\frac{1}{2}\dim Z + \frac{1}{2}} \; \left( \left( T_X^{r_Z+d} \; \frac{r_Z^{r_Z} \; d^d}{(r_Z+d)^{r_Z+d}} \right)^{\frac{1}{r_Z}} \right)^{-\frac{1}{2} \dim Z} \\
& \hspace{19em} \cdot \frac{1}{\Gamma(1+dx)} \cdot \left( A_X + \mathit{o}(1) \right).
\end{align*}
\endgroup
Therefore, when $r_Z=r_X-d>1$, $t^{\frac{1}{2}\dim Z}J^Z(t)$ is asymptotically Mittag-Leffler, with scaling $(T,A)$ given by:
\begin{equation*}
T_Z = \left( T_X^{r_X} \; \frac{r_Z^{r_Z} \; d^d}{r_X^{r_X}} \right)^{\frac{1}{r_Z}}, \quad \left( \text{or } \left( \frac{T_Z}{r_Z} \right)^{r_Z} = d^d \; \left( \frac{T_X}{r_X} \right)^{r_X} \right)
\end{equation*}
and
\begin{equation*}
A_Z = \left( 2\pi \; \frac{r_Z \; d}{r_X} \right)^\frac{1}{2} \; \frac{r_Z^{\frac{1}{2}\dim Z}}{T_Z^{\frac{1}{2} \dim Z}} \; \frac{T_X^{\frac{1}{2}\dim X}}{r_X^{\frac{1}{2}\dim X}} \cdot \frac{A_X}{\Gamma(1+dx)}.
\end{equation*}


\subsubsection*{(Case 2: $r_Z=r_X-d=1$)}

In the case where $r_Z=r_X-d=1$, we have
\begin{equation*}
J_{m} = \sum_{m_1+m_2=m} \frac{(-c_0)^{m_1}}{\Gamma(1+m_1)} \cdot \frac{\Gamma(1+dm_2+dx)}{\Gamma(1+dx)} J^{X}_{r_X m_2}.
\end{equation*}
Applying Proposition \ref{propDisGammaIProduct} directly, the result follows from the computations performed in the $r_Z=r_X-d>1$ case. It cn be shown that $t^{\frac{1}{2}\dim Z} J^Z(t)$ is also asymptotically Mittag-Leffler, with scaling $(T,A)$ given by:
\begin{equation*}
T_Z =  d^d \; \left( \frac{T_X}{r_X} \right)^{r_X}-c_0,
\end{equation*}
and 
\begin{equation*}
A_Z = \left( 2\pi \; \frac{d}{r_X} \right)^\frac{1}{2} \; \frac{1}{T_Z^{\frac{1}{2} \dim Z}} \; \frac{T_X^{\frac{1}{2}\dim X}}{r_X^{\frac{1}{2}\dim X}} \cdot \frac{A_X}{\Gamma(1+dx)}.
\end{equation*}

\end{proof}
}

\newpage

\section*{}

\printbibliography

\end{document}